\documentclass{article}
\usepackage{authblk,amsfonts,amsmath,amssymb,amsthm,dsfont,fancyhdr,}
\usepackage{graphicx}
\usepackage{tikz} 
\usepackage{latexsym}
\usepackage{gensymb}
\usepackage{fullpage}
\usepackage{enumerate}
\usepackage{enumitem}
\usepackage{xcolor}
\usepackage{caption}
\usepackage{float}
\usepackage[titletoc]{appendix}
\usetikzlibrary[trees]
\usepackage{array}
\usepackage{lipsum}
\usepackage{mathrsfs}
\usepackage{indentfirst}
\usepackage[hyphens]{url}
\usepackage[colorlinks=true,, allcolors=blue]{hyperref}

\newcommand\keywords[1]{\textbf{Keywords}: #1}
\newcommand\MSC[1]{\textbf{MSC2020}: #1}

\newtheorem{theorem}{Theorem}[section]
\newtheorem{corollary}[theorem]{Corollary}
\newtheorem{lemma}[theorem]{Lemma}
\newtheorem{remark}[theorem]{Remark}

\newtheorem{proposition}[theorem]{Proposition}
\newtheorem{definition}[theorem]{Definition}

\newtheorem{problem}[theorem]{Problem}

\usepackage{microtype}
\date{}
\begin{document}
\title{A structure theory for signed graphs with fixed smallest eigenvalue}

\author[a]{Meng-Yue Cao}
\author[a,b]{Jack H. Koolen}
\author[a]{Jing-Yuan Liu}
\author[c,d]{Qianqian Yang}
\affil[a]{\footnotesize{School of Mathematical Sciences, University of Science and Technology of China, Hefei, 230026, People's Republic of China}}
\affil[b]{\footnotesize{CAS Wu Wen-Tsun Key Laboratory of Mathematics, University of Science and Technology of China, Hefei, 230026, People's Republic of China}}
\affil[c] {\footnotesize{Department of Mathematics, Shanghai University, Shanghai 200444, People's Republic of China}}
\affil[d] {\footnotesize{Newtouch Center for Mathematics of Shanghai University, Shanghai 200444, People's Republic of China}}
\maketitle

	\newcommand\blfootnote[1]{%
		\begingroup
		\renewcommand\thefootnote{}\footnote{#1}%
		\addtocounter{footnote}{-1}%
		\endgroup}
	\blfootnote{ E-mail addresses: {\tt caomengyue@ustc.edu.cn} (M.-Y. Cao), {\tt koolen@ustc.edu.cn} (J.H. Koolen), {\tt liujingyuan@mail.ustc.edu.\newline cn} (J.-Y. Liu), {\tt  qqyang@shu.edu.cn} (Q. Yang)}
\vspace{-50pt}

\begin{abstract}
In this paper, we give a structure theory for signed graphs with fixed smallest eigenvalue. As a consequence, we prove that for every $\lambda\in(-1-\sqrt{2},-2]$, if a connected signed graph has smallest eigenvalue at least $\lambda$ and sufficiently large minimum valency, then its smallest eigenvalue is at least $-2$ and it is $1$-integrable. Thus, every such signed graph admits a representation by a family of $\{0,\pm1\}$-vectors of squared norm $2$ and may therefore be viewed as a natural signed generalization of generalized line graphs.
\end{abstract}
\newline
\MSC{05C50, 05C22, 52C35, 05C62, 05C75}

\section{Introduction}\label{sec intro}

  All graphs mentioned in this paper are finite, undirected, simple and non-empty. \par
A \emph{signed graph} $(G,\sigma)$ is a pair of a graph $G=(V(G),E(G))$ and a signing $\sigma:E(G)\rightarrow\{+,-\}$. We call $G$ the \emph{underlying graph} of $(G,\sigma)$. We say $(G,\sigma)$ is \emph{isomorphic} to $(H,\tau)$ if there exists an \emph{isomorphism} $\psi: G\rightarrow H$ such that $\sigma(\{x,y\})=\tau(\{\psi(x),\psi(y)\})$ for all $\{x,y\}\in E(G)$. The \emph{valency} of a vertex $x$ in $(G,\sigma)$ is the cardinality of the set $\{y\in V(G)\mid \{x,y\}\in E(G)\}$. The \emph{positive} (resp. \emph{negative}) \emph{graph} of $(G,\sigma)$ is the unsigned graph $(G,\sigma)^+$ (resp. $(G,\sigma)^-$) with vertex set $V(G)$ such that the edge set $E((G,\sigma)^+)=\{\{x,y\}\in E(G)\mid \sigma(\{x,y\})=+\}$ (resp. $E((G,\sigma)^-)=\{\{x,y\}\in E(G)\mid \sigma(\{x,y\})=-\}$). We mean by a \emph{subgraph} of a signed graph $(G,\sigma)$ on vertex set $U\subseteq V(G)$, the signed graph $(H,\sigma_H)$ with vertex set $U$ such that $E(H)\subseteq\binom{V(H)}{2}\cap E(G)$ and $\sigma_H=\sigma|_{E(H)}$. If $E(H)=\binom{V(H)}{2}\cap E(G)$, we say $(H,\sigma_H)$ is an \emph{induced subgraph} of $(G,\sigma)$. \par

Let $U\subseteq V(G)$. We say the signed graph $(G,\tau)$ is obtained from $(G,\sigma)$ by \emph{switching} with respect to $U$, if the signing $\tau$ satisfies the following: for any edge $\{x,y\}$ of $G$, $\tau (\{x,y\})=\sigma(\{x,y\})$ if $x,y$ are both in $U$ or both in $V(G)\backslash U$, and $\tau(\{x,y\})\neq\sigma(\{x,y\})$ otherwise. If $(G,\tau)$ can be obtained from $(G,\sigma)$ by switching, we say that $(G,\tau)$ and $(G,\sigma)$ are \emph{switching equivalent}. 
Given a signed graph $(H,\tau)$, we say the signed graph $(G,\sigma)$ is \emph{$(H,\tau)$-switching-free}, if $(G,\sigma)$ contains no induced subgraphs switching equivalent to $(H,\tau)$. \par

As in the case of unsigned graphs, we define the \emph{adjacency matrix} $A=A(G,\sigma)$ of $(G,\sigma)$ to be the symmetric matrix indexed by $V(G)$ such that 
 $$A_{xy}=\left\{\begin{array}{ll}
    1, & \text{if } \{x,y\}\in E(G)\text{ and }\sigma(\{x,y\})=+, \\
    -1, & \text{if } \{x,y\}\in E(G)\text{ and }\sigma(\{x,y\})=-, \\
   0,  & \text{otherwise}.
  \end{array}\right.$$
The \emph{eigenvalues} of $(G,\sigma)$ are the eigenvalues of the adjacency matrix $A$. \par

In this paper, we study signed graphs with a fixed smallest eigenvalue. We will see that several results for the unsigned case can be generalized to signed graphs.\par
 In order to state the results, we need to introduce several notations. For a graph $G$, we write $(G,+)$ the signed graph $(G,\sigma)$ such that $\sigma(\{x,y\})=+$ for all edges $\{x,y\}\in E(G)$, and $(G,-)$ the signed graph $(G,\sigma)$ such that $\sigma(\{x,y\})=-$ for all edges $\{x,y\}\in E(G)$.
 With $K_n$ we will denote the complete graph on $n$ vertices. Let $t$ be a positive integer. We define $\widetilde{K}_{2t}^{(\varepsilon)}$ for $\varepsilon\in\{0,-\}$ as the signed graph $(G,\sigma)$ on vertex set $\{0,1,\ldots,2t\}$ such that 
$$\{i,j\}\in E(G) \text{ and } \sigma(\{i,j\})=\left\{\begin{array}{ll}
   +, & \text{if }1\leq i,j\leq 2t, \\
 +, & \text{if }i=0,1\leq j\leq t, \\
   -,  & \text{if }i=0,t<j\leq 2t\text{ and }\varepsilon=-.
  \end{array}\right.$$
Later it will be clear why we use this notation. For examples of $\widetilde{K}_{2t}^{(\varepsilon)}$ for $t=2$, see Fig. \ref{fig1:examples of K2t tilde}. 

\begin{figure}[h]
    \centering
    \includegraphics[width=0.4\linewidth]{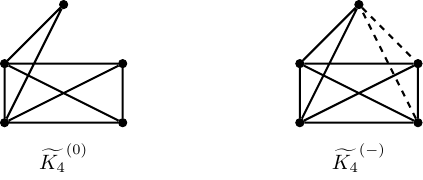}
    \caption{The signed graphs of $\widetilde{K_4}^{(0)}$ and $\widetilde{K_4}^{(-)}$}
    \vspace{-0.3cm}
    \caption*{\small\emph{(Edges with sign $+$ are
represented by solid segments and edges with sign $-$ are represented by dashed segments.)}}
    \label{fig1:examples of K2t tilde}
\end{figure}

A \emph{$t$-plex} is an unsigned graph such that each vertex has at most $t-1$ non-neighbors. 
Now we are able to state our first main result.
 \begin{theorem}\label{equaity of fobidden subgraphs and min ev}
  Let $(G,\sigma)$ be a signed graph with smallest eigenvalue $\lambda_{\min}(G,\sigma)$. The following hold.
  \begin{enumerate}[label=\rm{(\roman*)}]
    \item   For any real number $\lambda\leq -1$, there exists a positive integer $t=t(\lambda)$, such that if $\lambda_{\min}(G,\sigma)\geq\lambda$, then $(G,\sigma)$ is $\{\widetilde{K}_{2t}^{(0)},\widetilde{K}_{2t}^{(-)}$, $(K_{t+1},-)$, $(K_{1,t},+)\}$-switching-free.
    \item For any positive integer $t$, there exists a non-positive real number $\lambda=\lambda(t)$, such that if $(G,\sigma)$ is $\{\widetilde{K}_{2t}^{(0)},\widetilde{K}_{2t}^{(-)}$, $(K_{t+1},-)$, $(K_{1,t},+)\}$-switching-free, then $\lambda_{\min}(G,\sigma)\geq\lambda$.
\end{enumerate}
\end{theorem}

\begin{remark}\label{remark main thm 1.1}
  \begin{enumerate}[label=\rm{(\roman*)}]
  \item     Note that for any signed graph in the switching class of $\widetilde{K}_{2t}^{(0)},\widetilde{K}_{2t}^{(-)}$, $(K_{t+1},-)$ or $(K_{1,t},+)$, its smallest eigenvalue goes to $-\infty$ when $t$ goes to $+\infty$. This implies that $\rm{(i)}$ follows from the interlacing theorem \cite[Theorem $9.9.1$]{godsil2013}. We will give a proof of $\rm{(ii)}$ in Section \ref{sec of pf thm}. 
  \item Hoffman \cite{hoffman1977onsigned} proved an equivalent version of Theorem \ref{equaity of fobidden subgraphs and min ev} $\rm{(i)}$ using a different formulation, in which the $\{\widetilde{K}_{2t}^{(0)},\widetilde{K}_{2t}^{(-)}, (K_{t+1},-),(K_{1,t},+)\}$-switching-free condition is replaced by ten particular signed graphs.
\end{enumerate}
\end{remark}

We also show the following refinement of Theorem \ref{equaity of fobidden subgraphs and min ev}.

\begin{theorem}\label{structure theory thm}
 Let $\lambda\leq -1$ be a real number. There exists a positive integer $d_\lambda$ such that if $(G,\sigma)$ is a signed graph with smallest eigenvalue at least $\lambda$ and minimum valency at least $d_\lambda$, then there exists a set of induced subgraphs $N_1,N_2,\ldots,N_r$ of $(G,\sigma)$, where $r$ is a positive integer, satisfying the following conditions.
 \begin{enumerate}[label=\rm{(\roman*)}]
   \item Each vertex of $(G,\sigma)$ lies in at least one and at most $\lfloor-\lambda\rfloor$ $N_i$'s. 
   \item The induced subgraph $N_i$ is switching equivalent to a signed graph whose positive graph is a $(\lfloor \lambda^2+2\lambda+2\rfloor)$-plex, for $i=1,2,\ldots,r$.
   \item The intersection $V(N_i)\cap V(N_j)$ contains at most $4\lfloor-\lambda\rfloor-4$ vertices for $1\leq i< j\leq r$.
   \item the subgraph $(G',\sigma')$ has maximum valency at most $d_\lambda-1$, where $G'= (V(G),E(G)\backslash\bigcup_{i=1}^r E(N_i))$ and $\sigma'=\sigma|_{E(G')}$.
\end{enumerate} \end{theorem}

\begin{remark}
    Note that for the unsigned case, Theorem \ref{equaity of fobidden subgraphs and min ev} was shown by Hoffman \cite{hoffman1973on} and Theorem \ref{structure theory thm} was shown by Kim et al. \cite{kim2016a}.     \end{remark}

Gavrilyuk et al. \cite{Gavrilyuk2021signed} showed the following result.
\begin{theorem}[{\cite[Theorem $1.2$]{Gavrilyuk2021signed}}]\label{GMST result on -2}
 There exists an integer valued function $f$ defined on the half-open interval $(-2,-1]$ such that, for each $\lambda\in (-2, -1]$, if a connected signed graph $(G,\sigma)$ has smallest eigenvalue at least $\lambda$ and minimum valency at least $f(\lambda) $, then $(G,\sigma)$ is switching equivalent to a complete graph (with each edge signed $+$) and hence $\lambda_{\min} (G,\sigma)=-1$.  
\end{theorem}
We extend this result to the interval $(-1-\sqrt{2},-1]$. To state the result, we need to define $s$-integrable signed graphs, analogously to the corresponding notion for integral lattices where $s$ is a positive integer. For more details of lattices, see Subsection \ref{subsec lattices}.  \par

Let $s$ be a positive integer. A signed graph $(G,\sigma)$ with smallest eigenvalue $\lambda_{\min}$ is \emph{$s$-integrable}, if there exists an integer-valued matrix $N$ such that $s(A+\lceil -\lambda_{\min}\rceil \mathbf{I})=N^TN$, where $A$ is the adjacency matrix of $G$. Note that $(G,\sigma)$ is $s$-integrable if and only if the integral lattice generated by the columns of $\frac{1}{\sqrt{s}}N$ is $s$-integrable. \par

Our result is: 

\begin{theorem}\label{thm for min ev >-1-sqrt{2}}
Let $\lambda$ be a real number in $(-1-\sqrt{2},-1]$. There exists a positive integer $d'_{\lambda}$ such that if a connected signed graph $(G,\sigma)$ has smallest eigenvalue $\lambda_{\min}(G,\sigma)\geq \lambda$ and minimum valency at least $d'_{\lambda}$, then $\lambda_{\min}(G,\sigma)\geq-2$ and $(G,\sigma)$ is $1$-integrable. \end{theorem}

\begin{remark}
 \begin{enumerate}[label=\rm{(\roman*)}]
     \item Note that Theorem \ref{thm for min ev >-1-sqrt{2}} states that for $-1-\sqrt{2}<\lambda\leq -2$, if a signed graph $(G,\sigma)$ has smallest eigenvalue $\lambda_{\min}(G,\sigma)\geq \lambda$ and large enough minimum valency, then its adjacency matrix $A(G,\sigma)$ satisfies that $A(G,\sigma)+2\mathbf{I}=N^TN$ for some $\{0,\pm 1\}$-matrix $N$.
     \item Theorem \ref{GMST result on -2} and Theorem \ref{thm for min ev >-1-sqrt{2}} were shown by Hoffman \cite{hoffman1977on} for the unsigned case, as a generalized line graph is an unsigned graph with adjacency matrix $A$ such that $A+2\mathbf{I}=N^TN$ for some $\{0,\pm 1\}$-matrix $N$.
     \item For $\lambda\geq-2$, Belardo et al. improved Theorem \ref{thm for min ev >-1-sqrt{2}} as follows:
\begin{theorem}[{\cite[Theorem $3.13$]{belardo2018open}}]\label{BCKW result on -2}
  Each connected signed graph $(G,\sigma)$ with smallest eigenvalue at least $-2$ is $2$-integrable. Moreover, if $(G,\sigma)$ has at least $121$ vertices, then $(G,\sigma)$ is $1$-integrable.
\end{theorem}
This result was probably known by Hoffman, as he studied signed graphs in \cite{hoffman1977onsigned}.

 \end{enumerate}
\end{remark}

\subsection{\texorpdfstring{Line systems, $\mathbb{F}_3$-weighted complete graphs and signed graphs}{Line systems, F3-weighted complete graphs and signed graphs}}

In this subsection, we look at line systems in the Euclidean space $\mathbb{R}^d$. Let $\mathcal{L}$ be a line system. Represent the lines in $\mathcal{L}$ by unit vectors $\mathbf{u}_1,\ldots\mathbf{u}_n$ and denote by $Gr$ the Gram matrix of these vectors.\par

If $\mathcal{L}$ is a set of equiangular lines, which means that there exists a real number $\alpha\neq0$ such that $(\mathbf{u}_i,\mathbf{u}_j)\in\{\pm\alpha\}$ for $1\leq i\neq j\leq n$, then the matrix $\frac{1}{\alpha}(Gr-\mathbf{I})$ is a symmetric matrix with only $\pm1$ off the diagonal. This is the so-called Seidel matrix, which was first introduced by Van Lint and Seidel in \cite{vanlint1966}. If the lines of $\mathcal{L}$ form two different angles, $\arccos{\alpha}$ and $90\degree$, then the matrix $\frac{1}{\alpha}(Gr-\mathbf{I})$ is a symmetric matrix with $0$ on the diagonal and $0$ or $\pm1$ off the diagonal, and we call it the \emph{generalized Seidel matrix} of $\{\mathbf{u}_1,\ldots\mathbf{u}_n\}$.

\begin{problem}\label{pro motivation}
    Given a positive real number $\alpha$ and a positive integer $d$, what is the maximal number $n$ such that there exists a line system of $n$ lines in the Euclidean space $\mathbb{R}^d$ with angles $\arccos{\alpha}$ and $90\degree$?
\end{problem}

It is known that for a positive integer $d$, we have $n\leq\binom{d+3}{3}$, as the set of vectors $\{\mathbf{u}_1,\ldots\mathbf{u}_n\}$ forms a spherical Euclidean $3$-distance set (see \cite{Blokhuis1984Few} and \cite{Bannai1983an}). In \cite{glazyrin2018upper}, an upper bound $\binom{d+1}{2}$ for the cardinality of a spherical Euclidean $2$-distance set was proved for most $d$. We guess that the bound for the cardinality of a spherical Euclidean $3$-distance set may be possible to be improved to $\binom{d+2}{3}$, and if so this will give an upper bound of $n$ in Problem \ref{pro motivation}. \par

In order to solve Problem \ref{pro motivation}, it is essential to study $\{0,\pm1\}$-matrices with fixed smallest eigenvalue. \par

Let $S$ be a symmetric matrix with $0$ on the diagonal and $0$ or $\pm1$ off the diagonal, and $\{u_1,\ldots,u_n\}$ be a set of vectors with Gram matrix $Gr=\alpha S+\mathbf{I}$. It is natural to regard the normalized inner products $\frac{1}{\alpha}(\mathbf{u}_i,\mathbf{u}_j)$ as the weight of the edge $\{i,j\}$ in $K_n$, by which we obtain a weighted complete graph $K_n$ with weights $\{0,\pm1\}$. It is reasonable to consider $\{0,\pm1\}$ as the field $\mathbb{F}_3$ with $3$ elements, hence we can consider $S$ as the adjacency matrix of a $\mathbb{F}_3$-weighted complete graph. Another way to look at $S$, see also \cite{Koolen2021recent,Balla2025equiangular} for similar discussions, is to regard $S$ as the adjacency matrix of a signed graph $(G,\sigma)$, where the vertex set is $\{1,\ldots,n\}$, the edge set is $\{\{i,j\}\mid (\mathbf{u}_i,\mathbf{u}_j)\neq0\}$, and the signing is $\sigma(\{i,j\})=+$ (resp. $\sigma(\{i,j\})=-$) if $\frac{1}{\alpha}(\mathbf{u}_i,\mathbf{u}_j)=+1$ (resp. $\frac{1}{\alpha}(\mathbf{u}_i,\mathbf{u}_j)=-1$). In the first viewpoint, we consider $0$ and $\pm1$ all the same, whereas in the second viewpoint, $0$ plays a different role from $\pm 1$.  \par

In this paper, we will mainly use the language of signed graphs, except in Section \ref{sec fwcg}.  \par

This paper is organized as follows. In Section \ref{sec fwcg}, we will define $\mathbb{F}_3$-weighted complete graphs and state some results; in Section \ref{sec sg and qpc}, we will introduce more definitions of signed graphs, and translate the results for $\mathbb{F}_3$-weighted complete graphs to signed graphs; in Section  \ref{sec Hsg}, we will introduce Hoffman signed graphs and associated Hoffman signed graphs; in Section \ref{sec of pf thm}, we will prove Theorem \ref{equaity of fobidden subgraphs and min ev} $\mathrm{(ii)}$ and Theorem \ref{structure theory thm}, and in Section \ref{sec of -1-sqrt2}, we will focus on signed graphs with smallest eigenvalue greater than $-1-\sqrt{2}$ and give a proof of Theorem \ref{thm for min ev >-1-sqrt{2}}. 

\section{\texorpdfstring{$\mathbb{F}_3$-weighted complete graphs}{F3-weighted complete graphs}}\label{sec fwcg}
\subsection{\texorpdfstring{Definitions of $\mathbb{F}_3$-weighted complete graphs}{Definitions of F3-weighted complete graphs}}\label{subsec def of fwg}

In this subsection, we introduce some definitions and notations of $\mathbb{F}_3$-weighted complete graphs. Recall that $\mathbb{F}_3$ is the field with elements $\{0,\pm1\}$.\par

 A \emph{$\mathbb{F}_3$-weighted complete graph} $(K,w)$ is a pair of a complete graph $K$ and a weight $w:\binom{V(K)}{2}\rightarrow \{0,\pm1\}$. 
We say $(K,w)$ is \emph{isomorphic} to $(K',w')$ if there exists an \emph{isomorphism} $\psi: V(K)\rightarrow V(K')$ such that $w(\{x,y\})=w'(\{\psi(x),\psi(y)\})$ for all $\{x,y\}\in \binom{V(K)}{2}$. We mean by the \emph{induced subgraph} of the $\mathbb{F}_3$-weighted complete graph $(K,w)$ on vertex set $U\subseteq V(K)$, the $\mathbb{F}_3$-weighted complete graph $(K_1,w_{K_1})$ such that $ V(K_1)=U$, $E(K_1)=\binom{U}{2}$ and $w_{K_1}=w|_{E(K_1)}$.
We denote by $(K,w)_U$ the induced subgraph of $(K,w)$ on $U\subseteq V(K)$.\par
Let $\varepsilon\in\{0,\pm1\}$. For any pair of vertices $x,y$ of $K$, we say $x$ and $y$ are \emph{$(\varepsilon)$-neighbors}, when $w(\{x,y\})=\varepsilon$. 
Let $x$ be a vertex of the $\mathbb{F}_3$-weighted complete graph $(K,w)$. The \emph{$(\varepsilon)$-neighborhood} $N^{(\varepsilon)}(x)$ of $x$ in $(K,w)$ is the set of $(\varepsilon)$-neighbors of $x$. 
 For a positive integer $t$ and a vertex subset $U\subseteq V(K)$, the \emph{$t$-$(\varepsilon)$-neighborhood} of $U$ in $(K,w)$, denoted as $N^{(\varepsilon)}_{(t)}(U)$, is the vertex set $\{x\in V(K)\mid x$ has at least $t$ $(\varepsilon)$-neighbors in $U\}$. We call the $\mathbb{F}_3$-weighted complete graph $(K,w)$ an \emph{$(\varepsilon)$-clique}, and denoted by $(K,\varepsilon)$, if $w(\{x,y\})=\varepsilon$ for all $\{x,y\}\in \binom{V(K)}{2}$.
 \par

For the $\mathbb{F}_3$-weighted complete graph $(K,w)$, we define the \emph{adjacency matrix} $A=A(K,w)$ to be the symmetric matrix indexed by $V(K)$ such that 
 $$A_{xy}=\left\{\begin{array}{ll}
    w(\{x,y\}), & \text{if } x\neq y, \\
   0,  & \text{if }x=y.
  \end{array}\right.$$
The \emph{eigenvalues} of $(K,w)$ are the eigenvalues of its adjacency matrix. 
\par

We say the $\mathbb{F}_3$-weighted complete graph $(K,w')$ is obtained from $(K,w)$ by \emph{switching} with respect to a vertex subset $U\subseteq V(K)$, if $w'$ satisfies the following: for any edge $\{x,y\}$ of $K$, $w'(\{x,y\})=w(\{x,y\})$ if $x,y$ are both in $U$ or both in $V(K)\backslash U$, and $w'(\{x,y\})=-w(\{x,y\})$ otherwise. If $(K,w')$ can be obtained from $(K,w)$ by switching, we say that $(K,w)$ and $(K,w')$ are \emph{switching equivalent}. Note that each switching equivalent $\mathbb{F}_3$-weighted complete graphs are cospectral.
Given a $\mathbb{F}_3$-weighted complete graph $(H,w_H)$, we say the $\mathbb{F}_3$-weighted complete graph $(K,w)$ is \emph{$(H,w_H)$-switching-free}, if $(K,w)$ contains no induced subgraphs that are switching equivalent to $(H,w_H)$. \par

By considering the non-edges in signed graphs as $(0)$-edges in $\mathbb{F}_3$-weighted complete graphs, a signed graph can be regarded as a $\mathbb{F}_3$-weighted complete graph. We denote by $\widetilde{K}_{2m}^{(0)},\widetilde{K}_{2m}^{(-1)}$ and $(K_{1,t},+1)$ the corresponding $\mathbb{F}_3$-weighted complete graphs of  signed graphs $\widetilde{K}_{2m}^{(0)},\widetilde{K}_{2m}^{(-)}$ and $(K_{1,t},+)$, respectively. Now it is clear why we use the notations for signed graphs $\widetilde{K}_{2m}^{(\varepsilon)}$ for $\varepsilon\in\{0,-\}$.
\par

By considering the weights as colors, the Ramsey theorem in three colors case can be applied to $\mathbb{F}_3$-weighted complete graph.
\begin{theorem}
[{\cite{ramsey1930}}]{\label{Ramseynumber}}
 Let $m, s$ and $ t$ be three positive integers. There exists a minimum positive integer $R(m,s,t)$ such that for any $\mathbb{F}_3$-weighted complete graph with at least $R(m,s,t)$ vertices, it contains a $(+1)$-clique $(K_m,+1)$, a $(-1)$-clique $(K_s,-1)$, or a $(0)$-clique $(K_t,0)$.
\end{theorem}

\subsection{\texorpdfstring{Results on $\mathbb{F}_3$-weighted complete graphs}{Results on F3-weighted complete graphs}}\label{subsec results on fwcg}
In this subsection, we show some properties on $\{\widetilde{K}_{2m}^{(0)},\widetilde{K}_{2m}^{(-1)}\}$-switching-free $\mathbb{F}_3$-weighted complete graphs for integer $m\geq2$.

\begin{lemma}
    \label{neighborhood of 3m-2-clique}
    Let $m\geq 2,n\geq 3m-2$ be integers and $(K,w)$ be a $\{\widetilde{K}_{2m}^{(0)},\widetilde{K}_{2m}^{(-1)}\}$-switching-free $\mathbb{F}_3$-weighted complete graph. Let $C$ be a $(+1)$-clique of $(K,w)$ with at least $n$ vertices. For any vertex $x\in V(K)$, one of the following holds.
    \begin{enumerate}[label=\rm{(\roman*)}]
        \item $x\in N^{(+1)}_{(m)}(V(C))$, and $x$ has at most $m-1$ $(-1)$-neighbors and at most $m-1$ $(0)$-neighbors in $C$.
        \item $x\in N^{(-1)}_{(m)}(V(C))$, and $x$ has at most $m-1$ $(+1)$-neighbors and at most $m-1$ $(0)$-neighbors in $C$.
        \item $x\in N^{(0)}_{(m)}(V(C))$, and $x$ has at most $m-1$ $(+1)$-neighbors and at most $m-1$ $(-1)$-neighbors in $C$.
    \end{enumerate}
\end{lemma}
\begin{proof}
 For any vertex $x\in V(K)$, at least one of the three sets $N^{(+1)}(x)\cap V(C)$, $N^{(-1)}(x)\cap V(C)$ and $N^{(0)}(x)\cap V(C)$ contains at least $m$ vertices, as $n\geq 3m-2$. Without loss of generality, we may assume that $|N^{(+1)}(x)\cap V(C)|\geq m$, then $x$ is contained in the set $N_{(m)}^{(+1)}(V(C))$. Since $(K,w)$ is $\{\widetilde{K}_{2m}^{(0)},\widetilde{K}_{2m}^{(-1)}\}$-switching-free, the vertex $x$ has at most $m-1$ $(-1)$-neighbors and at most $m-1$ $(0)$-neighbors in $C$. This completes the proof.
\end{proof}

\begin{lemma}\label{C-D lemma}
    Let $m\geq 2,n\geq2(m^2+m)$ be integers and $(K,w)$ be a $\{\widetilde{K}_{2m}^{(0)},\widetilde{K}_{2m}^{(-1)}\}$-switching-free $\mathbb{F}_3$-weighted complete graph. Let $C$ be a $(+1)$-clique of $(K,w)$ with at least $n$ vertices. Let $\varepsilon\in\{0,\pm1\}$, $\varepsilon'\in\{\pm1\}$. For any vertex $x\in V(K)$, if there exists a $(+1)$-clique $D$ with $m$ vertices such that $V(D)\subseteq N^{(\varepsilon)}(x)\cap N_{(m)}^{(\varepsilon')}(V(C))$, then $x\in N_{(m)}^{(\varepsilon\cdot\varepsilon')}(V(C))$.
\end{lemma}

  \begin{proof}
First, we assume that $\varepsilon'=+1$, then $V(D)\subseteq N^{(\varepsilon)}(x)\cap N_{(m)}^{(+1)}(V(C))$. Since $V(D)\subseteq N_{(m)}^{(+1)}(V(C))$, each vertex of $D$ has at most $2m-2$ non-$(+1)$-neighbors in $C$, by Lemma \ref{neighborhood of 3m-2-clique}. Thus, $C$ contains at least $n-(2m-2)m\geq 2(m^2+m)-2m^2+2m=4m$ vertices which are $(+1)$-neighbors of each vertex in $D$. This implies that $|V(D)\cup (V(C)\cap N_{(m)}^{(+1)}(V(D)))|\geq|V(C)\cap N_{(m)}^{(+1)}(V(D))|\geq 4m$. Note that the subgraph induced on $V(D)\cup (V(C)\cap N_{(m)}^{(+1)}(V(D)))$ is a $(+1)$-clique. As $V(D)\subseteq N^{(\varepsilon)}(x)$, the vertex $x$ has at least $4m-(2m-2)\geq 2m+2$ $(\varepsilon)$-neighbors in $V(D)\cup (V(C)\cap N_{(m)}^{(+1)}(V(D)))$, by Lemma \ref{neighborhood of 3m-2-clique}. Hence, the vertex $x$ has at least $2m+2-m>m$ $(\varepsilon)$-neighbors in $C$, in other words, the vertex $x\in N_{(m)}^{(\varepsilon)}(V(C))$.\par

Now we assume that $\varepsilon'=-1$, then $V(D)\subseteq N^{(\varepsilon)}(x)\cap N_{(m)}^{(-1)}(V(C))$. Since $V(D)\subseteq N_{(m)}^{(-1)}(V(C))$, each vertex of $D$ has at most $2m-2$ non-$(-1)$-neighbors in $C$, by Lemma \ref{neighborhood of 3m-2-clique}. Thus, $C$ contains at least $n-(2m-2)m\geq 2(m^2+m)-2m^2+2m=4m$ vertices which are $(-1)$-neighbors of each vertex in $D$. We claim that $x$ has at least $m$ $(-\varepsilon)$-neighbors in $C$. Suppose that $x\not\in N_{(m)}^{(-\varepsilon)}(V(C))$, then $x$ has at most $m-1$ $(-\varepsilon)$-neighbors in $C$, by Lemma \ref{neighborhood of 3m-2-clique}. Thus,  the vertex $x$ has at least $4m-(m-1)\geq 3m+1$ non-$(-\varepsilon)$-neighbors in $V(C)\cap N_{(m)}^{(-1)}(V(D))$. By Lemma \ref{neighborhood of 3m-2-clique}, there exists $\eta\in\{0,\pm1\},\eta\neq-\varepsilon$, such that there are at least $m$ $(\eta)$-neighbors of $x$ in $V(C)\cap N_{(m)}^{(-1)}(V(D))$. Let $U$ be a vertex set of cardinality $m$ such that $U\subseteq V(C)\cap N_{(m)}^{(-1)}(V(D))\cap N^{(\eta)}(x)$. Let $(H,w_H)$ denote the subgraph of $(K,w)$ induced on $\{x\}\cup U\cup V(D)$, then $(H,w_H)$ is switching equivalent to $\widetilde{K}_{2m}^{(0)}$ or $\widetilde{K}_{2m}^{(-1)}$, which is a contradiction. Therefore, we obtain that $x\in N_{(m)}^{(-\varepsilon)}(V(C))$. This completes the proof.   
\end{proof}

Let $m\geq 2,n\geq 2(m^2+m)$ be integers and $(K,w)$ be a $\{\widetilde{K}_{2m}^{(0)},\widetilde{K}_{2m}^{(-1)}\}$-switching-free $\mathbb{F}_3$-weighted complete graph. Let $\mathfrak{C}(n)$ denote the set $\{C\mid C$  is a maximal $(+1)$-clique of $(K,w)$ with at least $n$ vertices$\}$. 
Let $C,C'\in \mathfrak{C}(n)$, the following lemma shows the relation between the neighborhoods of $C$ and $C'$.

\begin{lemma}\label{two cliques}
    Let $m\geq 2,n\geq 2(m^2+m)$ be integers and $(K,w)$ be a $\{\widetilde{K}_{2m}^{(0)},\widetilde{K}_{2m}^{(-1)}\}$-switching-free $\mathbb{F}_3$-weighted complete graph. For any $C,C'\in\mathfrak{C}(n)$, one of the following holds.
\begin{enumerate}[label=\rm{(\roman*)}]
    \item  There exists an  $\varepsilon\in\{\pm1\}$, such that $|V(C')\cap N_{(m)}^{(\varepsilon)}(V(C))|\geq m$, and $N_{(m)}^{(+1)}(V(C'))= N_{(m)}^{(\varepsilon)}(V(C))$, $N_{(m)}^{(-1)}(V(C'))=N_{(m)}^{(-\varepsilon)}(V(C))$.

    \item  $|V(C')\cap N_{(m)}^{(+1)}(V(C))|\leq m-1$, $|V(C')\cap N_{(m)}^{(-1)}(V(C))|\leq m-1$, and at most $2m-2$ vertices of $C'$ are not contained in $N_{(m)}^{(0)}(V(C))$.
\end{enumerate}
\end{lemma}
\begin{proof}
By Lemma \ref{neighborhood of 3m-2-clique}, for any vertex of $(K,w)$, it is contained in exactly one of the three vertex subsets $N_{(m)}^{(+1)}(V(C))$, $ N_{(m)}^{(-1)}(V(C))$ and $N_{(m)}^{(0)}(V(C))$. If $|V(C')\cap N_{(m)}^{(+1)}(V(C))|\leq m-1$ and $|V(C')\cap N_{(m)}^{(-1)}(V(C))|\leq m-1$, then $\rm{(ii)}$ holds. \par

Now we may assume that $|V(C')\cap N_{(m)}^{(\varepsilon)}(V(C))|\geq m$ for some  $\varepsilon\in\{\pm1\}$. We claim that under this condition, $N_{(m)}^{(\varepsilon')}(V(C'))\subseteq N_{(m)}^{(\varepsilon\cdot\varepsilon')}(V(C))$ holds for any $\varepsilon'\in \{\pm 1\}$. Take $m$ vertices from $V(C')\cap N_{(m)}^{(\varepsilon)}(V(C))$, then these vertices induce a $(+1)$-clique, denoted by $D$. Note that $V(D)\subseteq N^{(+1)}(x)$ for any vertex $x\in V(C')\backslash V(D)$. By Lemma \ref{C-D lemma}, we obtain that $x\in N_{(m)}^{(\varepsilon)}(V(C))$. Thus, $V(C')\subseteq N_{(m)}^{(\varepsilon)}(V(C))$. For any vertex $y\in N_{(m)}^{(\varepsilon')}(V(C'))$, it has at least $m$ $(\varepsilon')$-neighbors in $C'$. This implies that there exists a $(K_m,+1)$, in $N^{(\varepsilon')}(y)\cap V(C')\subseteq N^{(\varepsilon')}(y)\cap N_{(m)}^{(\varepsilon)}(V(C))$.  By Lemma \ref{C-D lemma}, we obtain that $y\in N_{(m)}^{(\varepsilon\cdot\varepsilon')}(V(C))$. Hence, $N_{(m)}^{(\varepsilon')}(V(C'))\subseteq N_{(m)}^{(\varepsilon\cdot\varepsilon')}(V(C))$.\par

To complete the proof, it suffices to show that $N_{(m)}^{(\varepsilon')}(V(C))\subseteq N_{(m)}^{(\varepsilon\cdot\varepsilon')}(V(C'))$ for any $\varepsilon'\in \{\pm 1\}$. This follows once we prove that $|V(C)\cap N_{(m)}^{(\varepsilon)}(V(C'))|\geq m$, by the claim in the previous paragraph.
Since $V(D)\subseteq N_{(m)}^{(\varepsilon)}(V(C))$, each vertex of $D$ has at most $2m-2$ non-$(\varepsilon)$-neighbors in $C$. 
By Lemma \ref{neighborhood of 3m-2-clique}, there exist at least $n-(2m-2)m\geq 2(m^2+m)-2m^2+2m=4m$ vertices in $V(C)\cap N_{(m)}^{(\varepsilon)}(V(D))$. Take $m$ vertices from this set, then these vertices induce a $(+1)$-clique, denoted by $D'$. Note that $V(D')$ is also contained in $ N_{(m)}^{(\varepsilon)}(V(C'))$, as $V(D)\subseteq V(C')$. Thus, $|V(C)\cap N_{(m)}^{(\varepsilon)}(V(C'))|\geq m$. 
This completes the proof.
\end{proof}

\begin{lemma}\label{Ramsey in neighborhood}
    Let $m\geq2, s\geq 3, t\geq 2, n\geq(2m-1)t+1$ be integers. There exists an integer $\kappa_0=\kappa_0(m,s,t)$ such that for any $\{(K_s,-1)$, $(K_{1,t},+1)\}$-switching-free $\mathbb{F}_3$-weighted complete graph $(K,w)$ and any $C\in\mathfrak{C}(n)$, $\varepsilon\in\{\pm1\}$, if $U\subseteq N_{(m)}^{(\varepsilon)}(V(C))$ is a vertex subset with at least $\kappa_0$ vertices, then the subgraph induced on $U$ contains a $(+1)$-clique $(K_m,+1)$.
\end{lemma}
\begin{proof}
 Let $\kappa_0=\kappa_0(m,s,t):=R(m,s,t)$, as defined in Theorem \ref{Ramseynumber}. Let $(H,w_H)$ denote the subgraph induced on $U$. It follows from Lemma \ref{neighborhood of 3m-2-clique} that each vertex of $(H,w_H)$ has at most $2m-2$ non-$(\varepsilon)$-neighbors in $C$. By Theorem \ref{Ramseynumber}, $(H,w_H)$ contains a $(K_m,+1)$, $(K_s,-1)$, or a $(K_t,0)$. If $(H,w_H)$ contains a $(K_t,0)$, then there exists a vertex of $C$ adjacent to each vertex of this $(K_t,0)$, as $n-(2m-2)t-t\geq 1$. This contradicts that $(K,w)$ is $(K_{1,t},+1)$-switching-free. Since $(K,w)$ is $(K_s,-1)$-switching-free, we obtain that $(H,w_H)$ contains a $(K_m,+1)$. This completes the proof.
\end{proof}

\begin{lemma}\label{before quasi-cli sw to almost posi-clique}
  Let $m\geq2, s\geq 3, t\geq 2, n\geq \max\{2(m^2+m),(2m-1)t+1\}$ be integers. There exists a positive integer $\kappa_1=\kappa_1(m,s,t)$ such that for any $\{\widetilde{K}_{2m}^{(0)},\widetilde{K}_{2m}^{(-1)}$, $(K_s,-1)$, $(K_{1,t},+1)\}$-switching-free $\mathbb{F}_3$-weighted complete graph $(K,w)$, the following hold.
  \begin{enumerate}[label=\rm{(\roman*)}]
      \item For any $x\in V(K), C\in\mathfrak{C}(n)$.
    
      \begin{enumerate}
        [label=\rm{(\alph*)}]
            \item If $x\in N_{(m)}^{(+1)}(V(C))$, then $$|N_{(m)}^{(+1)}(V(C))\cap N^{(0)}(x)|\leq\kappa_1,\ \ |N_{(m)}^{(+1)}(V(C))\cap N^{(-1)}(x)|\leq\kappa_1,$$
            $$|N_{(m)}^{(-1)}(V(C))\cap N^{(0)}(x)|\leq\kappa_1,\ \ |N_{(m)}^{(-1)}(V(C))\cap N^{(+1)}(x)|\leq \kappa_1.$$
            \item If $x\in N_{(m)}^{(-1)}(V(C))$, then $$|N_{(m)}^{(+1)}(V(C))\cap N^{(0)}(x)|\leq\kappa_1,\ \  |N_{(m)}^{(+1)}(V(C))\cap N^{(+1)}(x)|\leq\kappa_1,$$ $$|N_{(m)}^{(-1)}(V(C))\cap N^{(0)}(x)|\leq\kappa_1,\ \ |N_{(m)}^{(-1)}(V(C))\cap N^{(-1)}(x)|\leq \kappa_1.$$
        \end{enumerate}
      \item For any $C_1,C_2\in\mathfrak{C}(n)$, if $N^{(+1)}_{(m)}(V(C_1))\cup N^{(-1)}_{(m)}(V(C_1))\neq N^{(+1)}_{(m)}(V(C_2))\cup N^{(-1)}_{(m)}(V(C_2))$, then $$|N_{(m)}^{(\varepsilon_1)}(V(C_1))\cap N_{(m)}^{(\varepsilon_2)}(V(C_2))|\leq\kappa_1$$ for any $\varepsilon_1,\varepsilon_2\in\{\pm1\}$.
  \end{enumerate}
\end{lemma}
\begin{proof}
Let $\kappa_1=\kappa_1(m,s,t):=\kappa_0(m,s,t)-1$, as defined in Lemma \ref{Ramsey in neighborhood}.\par
(i) Suppose that $x\in N_{(m)}^{(\varepsilon_0)}(V(C))$, where $\varepsilon_0\in\{\pm1\}$. Let $\varepsilon\in\{\pm1\}$ and $\varepsilon'\in\{0,-\varepsilon_0\cdot\varepsilon\mid \varepsilon_0,\varepsilon\in\{\pm1\}\}$. Suppose for contradiction that $x$ has at least $\kappa_1+1$ $(\varepsilon')$-neighbors in $N_{(m)}^{(\varepsilon)}(V(C))$. Let $(H_1,w_{H_1})$ be the subgraph induced on $N_{(m)}^{(\varepsilon)}(V(C))\cap N^{(\varepsilon')}(x)$. By Lemma \ref{Ramsey in neighborhood}, the graph $(H_1,w_{H_1})$ contains a $(K_m,+1)$. By Lemma \ref{C-D lemma}, the vertex $x\in N_{(m)}^{(\varepsilon\cdot\varepsilon')}(V(C))$. Note that $\varepsilon\cdot\varepsilon'\in\{0,-\varepsilon_0\}$, then $\varepsilon\cdot\varepsilon'\neq\varepsilon_0$, which contradicts that $x\in N_{(m)}^{(\varepsilon_0)}(V(C))$. This shows the statement $\rm{(i)}$ holds.\par
(ii) Suppose for contradiction that $N_{(m)}^{(\varepsilon_1)}(V(C_1))\cap N_{(m)}^{(\varepsilon_2)}(V(C_2))$ contains at least $\kappa_1+1$ vertices, for some $\varepsilon_1,\varepsilon_2\in\{\pm1\}$. Let $(H_2,w_{H_2})$ denote the subgraph induced on $N_{(m)}^{(\varepsilon_1)}(V(C_1))\cap N_{(m)}^{(\varepsilon_2)}(V(C_2))$. By Lemma \ref{Ramsey in neighborhood}, the graph $(H_2,w_{H_2})$ contains a $(K_m,+1)$, denote by $D$. By Lemma \ref{neighborhood of 3m-2-clique}, each vertex of $D$ has at most $2m-2$ non-$(\varepsilon_1)$-neighbors in $C_1$. Let $U:=\{x\in V(C_1)\mid V(D)\subseteq N^{(\varepsilon_1)}(x)\}$, then $|U|\geq n-(2m-2)m\geq 4m$. Since $V(D)\subseteq N_{(m)}^{(\varepsilon_2)}(V(C_2))$, we find that, by Lemma \ref{C-D lemma}, the vertex $x\in N_{(m)}^{(\varepsilon_1\cdot\varepsilon_2)}(V(C_2))$ for any $x\in U$. Thus, $|N_{(m)}^{(\varepsilon_1\cdot\varepsilon_2)}(V(C_2))\cap V(C_1)|\geq|U|\geq4m$. Since $\varepsilon_1,\varepsilon_2\in\{\pm1\}$ and $\varepsilon_1\cdot\varepsilon_2\in\{\pm1\}$, by Lemma \ref{two cliques}, we obtain that $N^{(+1)}_{(m)}(V(C_1))=N^{(\varepsilon_1\cdot\varepsilon_2)}_{(m)}(V(C_2))$ and $ N^{(-1)}_{(m)}(V(C_1))=N^{(-\varepsilon_1\cdot\varepsilon_2)}_{(m)}(V(C_2))$. This contradicts that $N^{(+1)}_{(m)}(V(C_1))\cup N^{(-1)}_{(m)}(V(C_1))\neq N^{(+1)}_{(m)}(V(C_2))\cup N^{(-1)}_{(m)}(V(C_2))$. Hence, $|N_{(m)}^{(\varepsilon_1)}(V(C_1))\cap N_{(m)}^{(\varepsilon_2)}(V(C_2))|\leq\kappa_1$  for any $\varepsilon_1,\varepsilon_2\in\{\pm1\}$.
\end{proof}

\section{Signed graphs and quasi-positive-cliques}\label{sec sg and qpc}

In this section, we first introduce more definitions and notations of signed graphs, then translate the results we obtained in Subsection \ref{subsec results on fwcg}  to signed graphs, by considering the $(0)$-edges in $\mathbb{F}_3$-weighted complete graphs as non-edges in signed graphs.

\subsection{More definitions on signed graphs}

Let $(G,\sigma)$ be a signed graph. For $\{x,y\}\in E(G)$, we say $x$ and $y$ are \emph{positively adjacent} or \emph{positive neighbors} if $\sigma(\{x,y\})=+$, and we say $x$ and $y$ are \emph{negatively adjacent} or \emph{negative neighbors} if $\sigma(\{x,y\})=-$. If $\{x,y\}\not\in E(G)$, we say $x$ and $y$ are \emph{non-adjacent} or \emph{non-neighbors}.
For $x\in V(G)$, the \emph{positive-neighborhood} $N^{(+)}(x)$ of $x$ in $(G,\sigma)$ (resp. the \emph{negative-neighborhood} $N^{(-)}(x)$ of $x$ in $(G,\sigma)$) is the set of the positive-neighbors (resp. negative neighbors) of $x$.
 For a positive integer $t$ and a vertex subset $U\subseteq V(G)$, the \emph{$t$-positive-neighborhood} $N^{(+)}_{(t)}(x)$ of $U$ in $(G,\sigma)$ (resp. the \emph{$t$-negative-neighborhood} $N^{(-)}_{(t)}(x)$ of $U$ in $(G,\sigma)$) is the set $\{x\in V(G)\mid x$ has at least $t$ positive neighbors (resp. negative neighbors) in $U\}$. \par
 The induced subgraph of $(G,\sigma)$ on $U\subseteq V(G)$ is denoted by $(G,\sigma)_U$. We call $(G,\sigma)_U$ a \emph{positive clique} (resp. \emph{negative clique}) if $(G,\sigma)_U$ is isomorphic to $(K_{|U|},+)$ (resp. $(K_{|U|},-)$).

\subsection{Quasi-positive-cliques}\label{subsec qpc}

In this subsection, we will define quasi-positive-cliques of $\{\widetilde{K}_{2m}^{(0)},\widetilde{K}_{2m}^{(-)}\}$-switching-free signed graphs, and show a few properties of quasi-positive-cliques. \par
First, we state two results on $\{\widetilde{K}_{2m}^{(0)},\widetilde{K}_{2m}^{(-)}\}$-switching-free signed graphs before we define quasi-positive-cliques. \par

\begin{lemma}
    \label{neighborhood of 3m-2-clique signed case}
    Let $m\geq 2,n\geq 3m-2$ be integers and $(G,\sigma)$ be a $\{\widetilde{K}_{2m}^{(0)},\widetilde{K}_{2m}^{(-)}\}$-switching-free signed graph. Let $C$ be a positive clique of $(G,\sigma)$ with at least $n$ vertices. For any vertex $x\in V(G)$, one of the following holds.
    \begin{enumerate}[label=\rm{(\roman*)}]
        \item $x\in N^{(+)}_{(m)}(V(C))$, and $x$ has at most $m-1$ negative neighbors and at most $m-1$ non-neighbors in $C$.
        \item $x\in N^{(-)}_{(m)}(V(C))$, and $x$ has at most $m-1$ positive neighbors and at most $m-1$ non-neighbors in $C$.
        \item  $x$ has at most $m-1$ positive neighbors and at most $m-1$ negative neighbors in $C$.
    \end{enumerate}
\end{lemma}
\begin{proof}
    It follows from Lemma \ref{neighborhood of 3m-2-clique} straightforward. 
\end{proof}
Let $m\geq 2,n\geq 2(m^2+m)$ be integers and $(G,\sigma)$ be a $\{\widetilde{K}_{2m}^{(0)},\widetilde{K}_{2m}^{(-)}\}$-switching-free signed graph. Let $\mathcal{C}(n)$ denote the set $\{C\mid C$  is a maximal positive clique of $(G,\sigma)$ with at least $n$ vertices$\}$. The following lemma shows that for any pair of positive cliques in $\mathcal{C}(n)$, either there are few edges between them, or they have the same $m$-neighborhood.
 
\begin{lemma}\label{two cliques signed case}
    Let $m\geq 2,n\geq 2(m^2+m)$ be integers and $(G,\sigma)$ be a $\{\widetilde{K}_{2m}^{(0)},\widetilde{K}_{2m}^{(-)}\}$-switching-free signed graph. For any $C,C'\in\mathcal{C}(n)$, one of the following holds.
    
\begin{enumerate}[label=\rm{(\roman*)}]
    \item There are at most $m-1$ vertices of $C'$ contained in $N_{(m)}^{(+)}(V(C))$, at most $m-1$ vertices of $C'$ contained in $N_{(m)}^{(-)}(V(C))$, and each of the other vertices of $C'$ has at most $2m-2$ neighbors in $C$. 
    \item $N_{(m)}^{(+)}(V(C'))\cup N_{(m)}^{(-)}(V(C'))= N_{(m)}^{(+)}(V(C))\cup N_{(m)}^{(-)}(V(C))$.
\end{enumerate}
\end{lemma}
\begin{proof}
    By Lemma \ref{two cliques}, if the statement $\rm{(i)}$ does not hold, then either $$N_{(m)}^{(+)}(V(C'))=N_{(m)}^{(+)}(V(C)),\ N_{(m)}^{(-)}(V(C'))=N_{(m)}^{(-)}(V(C))$$ or $$N_{(m)}^{(+)}(V(C'))=N_{(m)}^{(-)}(V(C)),\ N_{(m)}^{(-)}(V(C'))=N_{(m)}^{(+)}(V(C)).$$
    This completes the proof.
\end{proof}
For any $C\in \mathcal{C}(n)$, define the \emph{quasi-positive-clique $Q(C)$} of $(G,\sigma)$, with respect to the pair $(m,n)$, to be the signed graph obtained from $(G,\sigma)_{N_{(m)}^{(+)}(V(C))\cup N_{(m)}^{(-)}(V(C))}$ by switching on $N_{(m)}^{(-)}(V(C))$.

\begin{proposition}
    \label{quasi-cli sw to almost posi-clique}
  Let $m\geq2, s\geq 3, t\geq2,n\geq \max\{2(m^2+m),(2m-1)t+1\}$ be integers. There exists a positive integer $\kappa_2=\kappa_2(m,s,t)$, such that for any $\{\widetilde{K}_{2m}^{(0)},\widetilde{K}_{2m}^{(-)}$, $(K_s,-)$, $(K_{1,t},+)\}$-switching-free signed graph $(G,\sigma)$ and any $C\in\mathcal{C}(n)$, the positive graph of $Q(C)$ is a $\kappa_2$-plex, where $Q(C)$ is the quasi-positive-clique with respect to $C$.
\end{proposition}

\begin{proof}
Let $\kappa_2=\kappa_2(m,s,t):=4\kappa_1+1$, as defined in Lemma \ref{before quasi-cli sw to almost posi-clique}.\par
Let $x$ be a vertex of $Q(C)$. By Lemma \ref{before quasi-cli sw to almost posi-clique} $\rm{(i)}$, there are at most $2\kappa_1$ non-neighbors of $x$ in $Q(C)$. Now we count the negative neighbors of $x$ in $Q(C)$. 
If $x\in N^{(+)}_{(m)}(V(C))$, the set of negative neighbors of $x$ in $Q(C)$ is  
    $(N^{(+)}_{(m)}(V(C))\cap N^{(-)}(x))\cup (N^{(-)}_{(m)}(V(C))\cap N^{(+)}(x))$.
If $x\in N^{(-)}_{(m)}(V(C))$, the set of negative neighbors of $x$ in $Q(C)$ is  
    $(N^{(+)}_{(m)}(V(C))\cap N^{(+)}(x))\cup (N^{(-)}_{(m)}(V(C))\cap N^{(-)}(x))$.
By Lemma \ref{before quasi-cli sw to almost posi-clique} $\rm{(i)}$, there are at most $2\kappa_1$ negative neighbors of $x$ in $Q(C)$. 
Therefore, the positive graph of $Q(C)$ is a $\kappa_2$-plex. This completes the proof. 
\end{proof}

Lemma \ref{two cliques signed case} and the following lemma show that for any pair of positive cliques in $\mathcal{C}(n)$, their quasi-positive-cliques are either the same or have a bounded intersection.

\begin{lemma}\label{quasi-cli almost isolate}
   Let $m\geq 2, s\geq 3, t\geq 2,n\geq \max\{2(m^2+m),(2m-1)t+1\}$ be integers. There exists a positive integer $\kappa_3=\kappa_3(m,s,t)$, such that for any $\{\widetilde{K}_{2m}^{(0)},\widetilde{K}_{2m}^{(-)},(K_s,-)$, $(K_{1,t},+)\}$-switching-free signed graph $(G,\sigma
   )$ and any pair of  $C_1,C_2\in\mathcal{C}(n)$, if $V(Q(C_1))\neq V(Q(C_2))$, then $|N_{(m)}^{(\varepsilon_1)}(V(C_1))\cap N_{(m)}^{(\varepsilon_2)}(V(C_2))|\leq\kappa_3$ for any $\varepsilon_1,\varepsilon_2\in\{+,-\}$. \end{lemma}

\begin{proof}
Let $\kappa_3=\kappa_3(m,s,t):=\kappa_1$, as defined in Lemma \ref{before quasi-cli sw to almost posi-clique}. Since $V(Q(C_1))\neq V(Q(C_2))$, we have $N^{(+1)}_{(m)}(V(C_1))\cup N^{(-1)}_{(m)}(V(C_1))\neq N^{(+1)}_{(m)}(V(C_2))\cup N^{(-1)}_{(m)}(V(C_2))$. By Lemma \ref{before quasi-cli sw to almost posi-clique} $\rm{(ii)}$, we obtain that $|N_{(m)}^{(\varepsilon_1)}(V(C_1))\cap N_{(m)}^{(\varepsilon_2)}(V(C_2))|\leq\kappa_3$, for any $\varepsilon_1,\varepsilon_2\in\{+,-\}$.
\end{proof}

\section{Hoffman signed graphs}\label{sec Hsg}

In this section, we first introduce the definitions and key properties of Hoffman signed graphs, then construct the associated Hoffman signed graphs by means of the induced subgraphs which are switching equivalent to quasi-positive-cliques.

\subsection{Definitions of Hoffman signed graphs}

A \emph{Hoffman signed graph} $\mathfrak{h}=(H,\sigma,\ell)$  is a pair of a signed graph $(H,\sigma)$ and a labeling map $\ell:V(H)\rightarrow\{f,s\}$, satisfying that any pair of vertices with label $f$ are non-adjacent and any vertex with label $f$ has at least one neighbor. We call $(H,\sigma)$ the \emph{underlying signed graph} of $\mathfrak{h}$. We say a vertex with label $s$ a \emph{slim vertex}, and a vertex with label $f$ a \emph{fat vertex}. Denote by $V_s(\mathfrak{h})$ (resp. $V_f(\mathfrak{h})$) the \emph{slim vertex set} (resp. \emph{fat vertex set}) of $\mathfrak{h}$.
If each slim vertex of the Hoffman signed graph $\mathfrak{h}$ has at least one fat neighbor, we call $\mathfrak{h}$ \emph{fat}. 
 The \emph{slim graph} $(H_s,\sigma_s)$ of the Hoffman signed graph $\mathfrak{h}$ is the subgraph of $(H,\sigma)$ induced on $V_s(\mathfrak{h})$. We may consider a signed graph as a Hoffman signed graph with only slim vertices, and vice versa. For the rest part of the paper, we will not distinguish between them.\par
 
  A Hoffman signed graph $\mathfrak{h}_1=(H_1,\sigma_1,\ell_1)$ is called a \emph{subgraph} of $\mathfrak{h}=(H,\sigma,\ell)$, if $(H_1,\sigma_1)$ is a subgraph of $(H,\sigma)$ and $\ell_1=\ell|_{V(H_1)}$. And we say $\mathfrak{h}_1=(H_1,\sigma_1,\ell_1)$ is an \emph{induced subgraph} of $\mathfrak{h}=(H,\sigma,\ell)$, if $(H_1,\sigma_1)$ is an induced subgraph of $(H,\sigma)$ and $\ell_1=\ell|_{V(H_1)}$.
  Two Hoffman signed graphs $\mathfrak{h}=(H,\sigma,\ell)$ and $\mathfrak{h'}=(H',\sigma',\ell')$ are called \emph{isomorphic} if there exists an isomorphism between $(H,\sigma)$ and $(H',\sigma')$ that preserves the labeling.\par
  
  For a Hoffman signed graph $\mathfrak{h}=(H,\sigma,\ell)$, let
  $$A=\begin{pmatrix}
    A_s&C\\
    C^T&O
  \end{pmatrix}$$
  be the adjacency matrix of $(H,\sigma)$, where $A_s$ is the adjacency matrix of the slim graph of $(H,\sigma)$.
  The matrix $S(\mathfrak{h}):=A_s-CC^T$ is called the \emph{special matrix} of $\mathfrak{h}$. The \emph{eigenvalues} of a Hoffman signed graph are the eigenvalues of its special matrix. Denote by $\lambda_{\min}(\mathfrak{h})$ the smallest eigenvalue of the Hoffman signed graph $\mathfrak{h}$. Note that if a Hoffman signed graph $\mathfrak{h}=(H,\sigma,\ell)$ has only slim vertices, then $S(\mathfrak{h})=A(H)$, and thus the special eigenvalues and the adjacency eigenvalues are also the same. \par
  
We say that Hoffman signed graph $\mathfrak{h'}=(H',\sigma',\ell')$ can be obtained from $\mathfrak{h}=(H,\sigma,\ell)$ by \emph{switching} with respect to $U$ for some $U\subseteq V(H)$, if $(H',\sigma')$ is obtained from $(H,\sigma)$ by \emph{switching} with respect to $U$ and for any $x\in V(H')$, $\ell'(x)=\ell(x)$. We also call $\mathfrak{h}$ and $\mathfrak{h'}$ are \emph{switching equivalent}, if $\mathfrak{h}'$ can be obtained from $\mathfrak{h}$ by switching. 
\begin{lemma}\label{hsg sw}
Let $\mathfrak{h}=(H,\sigma,\ell)$ be a Hoffman signed graph. If  $\mathfrak{h}'=(H',\sigma',\ell')$ is switching equivalent to $\mathfrak{h}=(H,\sigma,\ell)$, then $\mathfrak{h'}$ has the same eigenvalues as $\mathfrak{h}$.  
\end{lemma}
\begin{proof}
Since $\mathfrak{h}=(H,\sigma,\ell)$ and $\mathfrak{h'}=(H',\sigma',\ell')$ are switching equivalent, we may assume that $(H',\sigma')$ is obtained from $(H,\sigma)$ by switching on $U$ for some $U\subseteq V(H)$. Define the diagonal matrix $D$ as
$$D_{xx}=\left\{\begin{array}{ll}
    -1, & \text{if } x\in U \\
    1,  & \text{otherwise}.
  \end{array}\right.$$
Note that $$\begin{pmatrix}
    A_s(H',\sigma')&C'\\
    C'^T&O
  \end{pmatrix}=A(H',\sigma')=D^{-1}A(H,\sigma)D=D^{-1}\begin{pmatrix}
    A_s(H,\sigma)&C\\
    C^T&O
  \end{pmatrix}D,$$
then $A_s(H',\sigma')=D_s^{-1}A_s(H,\sigma)D_s$ and $C'C'^T=D_s^{-1}CC^TD_s$, where $D_s$ is the submatrix of $D$ with rows and columns indexed by vertices in $V_s(\mathfrak{h})$. Thus, $S(\mathfrak{h'})=A_s(H',\sigma')-C'C'^T=D_s^{-1}A_s(H,\sigma)D_s-D_s^{-1}CC^TD_s=D_s^{-1}S(\mathfrak{h})D_s$. Hence, the eigenvalues of $\mathfrak{h}$ and $\mathfrak{h'}$ are the same.
\end{proof}

Let $\mathfrak{h}$ be a Hoffman signed graph and $n$ be a positive integer. We denote by $G(\mathfrak{h},n)$ the signed graph obtained by replacing each fat vertex of $\mathfrak{h}$ by a positive clique $(K_n,+)$. Hoffman in principle showed the following result in \cite{hoffman1977onsigned}, and Gavrilyuk et al. \cite{Gavrilyuk2021signed} reproved it and stated it in terms of Hoffman signed graphs.
\begin{theorem}[{\cite[Theorem $4.2$]{Gavrilyuk2021signed}}]\label{Hoffman-Ostrowski Theorem of signed case}
 For any Hoffman signed graph $\mathfrak{h}$ and any positive integer $n$, $\lambda_{\min}(G(\mathfrak{h},n))\geq\lambda_{\min}(\mathfrak{h})$, and $\lim\limits_{n\rightarrow+\infty}\lambda_{\min}(G(\mathfrak{h},n))=\lambda_{\min}(\mathfrak{h})$.
\end{theorem}

\subsection{Lattices}\label{subsec lattices}
Let $\mathbb{R}^n$ be an $n$-dimensional Euclidean space, equipped with the canonical \emph{inner product} $(\mathbf{x},\mathbf{y}):=\mathbf{x}^T\mathbf{y}$. The number $||\mathbf{x}||^2=(\mathbf{x},\mathbf{x})$ will be called the \emph{norm} of $\mathbf{x}\in\mathbb{R}^n$.\par

A \emph{lattice} $\Lambda$ in $\mathbb{R}^n$ is a discrete set of vectors in $\mathbb{R}^n$ which is closed under addition and subtraction. An \emph{integral lattice} is a lattice in which the inner product of any two vectors is integral. The \emph{direct sum} of two lattices $\Lambda_1$ and $\Lambda_2$ is defined if $\Lambda_1$ and $\Lambda_2$ are orthogonal, i.e., if $(\mathbf{x},\mathbf{y})=0$ for $\mathbf{x}\in\Lambda_1$, $\mathbf{y}\in\Lambda_2$ , as $\Lambda_1\oplus\Lambda_2=\{\mathbf{x}+\mathbf{y}\mid \mathbf{x}\in\Lambda_1,\mathbf{y}\in\Lambda_2\}$. A lattice $\Lambda$ is called \emph{irreducible} if $\Lambda=\Lambda_1\oplus\Lambda_2$ implies $\Lambda_1=\{\mathbf{0}\}$ or $\Lambda_2=\{\mathbf{0}\}$, and \emph{reducible} otherwise.\par
If $X$ is a set of vectors in  $\mathbb{R}^n$ such that $(\mathbf{x},\mathbf{y})\in \mathbb{Z}$ for all $\mathbf{x,y}\in X$, then $\Lambda=\{ \sum_{\mathbf{x}\in X}\alpha_{\mathbf{x}}\mathbf{x}\mid \alpha_{\mathbf{x}}\in\mathbb{Z}\}$ is an integral lattice. In this case, we say that the lattice $\Lambda$ is \emph{generated} by $X$, and denote it by $\langle X\rangle$.  We say an integral lattice $\Lambda$ has \emph{minimal norm} $t$ if $\min\{(\mathbf{x},\mathbf{x})\mid \mathbf{x}\neq\mathbf{0},\mathbf{x}\in \Lambda\}=t$.
\par

For a positive integer $s$, an integral lattice $\Lambda\subset \mathbb{R}^n$ is called \emph{$s$-integrable} if $\sqrt{s}\Lambda$ can be embedded in the standard lattice $\mathbb{Z}^k$ for some $k\geq n$. In other words, an integral lattice $\Lambda$ is $s$-integrable if and only if each vector in $\Lambda$ can be described by the form $\frac{1}{\sqrt{s}}(x_1,\ldots,x_k)^T$ with all $x_i\in \mathbb{Z}$ in $\mathbb{R}^k$ for some $k\geq n$ simultaneously. \par

\subsection{Hoffman signed graphs and Lattices}

We start with the representation and the reduced representation of Hoffman signed graphs.\par

\begin{definition}
   For a Hoffman signed graph $\mathfrak{h}$ and a positive integer $n$, a mapping 
 $\phi:V(\mathfrak{h})\rightarrow\mathbb{R}^n$ such that 
 $$(\phi(x),\phi(y))=\left\{\begin{array}{ll}
    m, & \text{if } x=y\in V_s(\mathfrak{h}), \\
    1, & \text{if } x=y\in V_f(\mathfrak{h}), \\
    1, & \text{if } x,y \text{ are positively adjacent},\\
    -1, & \text{if } x,y \text{ are negatively adjacent},\\
    0, & \text{otherwise},
  \end{array}\right.$$ 
is called a \emph{representation of $\mathfrak{h}$ of norm $m$}. 
\end{definition}
We denote by $\Lambda(\mathfrak{h},m)$ the lattice generated by $\{\phi(x)\mid x\in V(\mathfrak{h})\}$. Note that the isomorphism class of $\Lambda(\mathfrak{h},m)$ depends only on $m$, and is independent of $\phi$, justifying the notation.\par
 For a Hoffman signed graph $\mathfrak{h}=(H,\sigma,\ell)$, we define $$n_{\mathfrak{h}}^f(x,y):=\sum_{\substack{z\in V_f(\mathfrak{h})\\\{x,z\},\{y,z\}\in E(\mathfrak{h}) }}|\{z\mid \sigma(\{x,z\})=\sigma(\{y,z\})\}|-\sum_{\substack{z\in V_f(\mathfrak{h})\\ \{x,z\},\{y,z\}\in E(\mathfrak{h})}}|\{z\mid \sigma(\{x,z\})\neq\sigma(\{y,z\})\}|$$ for $x,y\in V_s(\mathfrak{h})$.
\begin{definition}
 For a Hoffman signed graph $\mathfrak{h}$ and a positive integer $n$, a mapping $\psi:V_s(\mathfrak{h})\rightarrow\mathbb{R}^n$ such that 
 $$(\psi(x),\psi(y))=\left\{\begin{array}{ll}
    m-n_{\mathfrak{h}}^f(x,y), & \text{if } x=y, \\
    1-n_{\mathfrak{h}}^f(x,y), & \text{if } x,y \text{ are positively adjacent},\\
    -1-n_{\mathfrak{h}}^f(x,y), & \text{if } x,y \text{ are negatively adjacent},\\
    -n_{\mathfrak{h}}^f(x,y), & \text{otherwise},
  \end{array}\right.$$ is called a \emph{reduced representation of $\mathfrak{h}$ of norm $m$}. 
\end{definition}
  We denote by $\Lambda^{\rm{red}}(\mathfrak{h},m)$ the lattice generated by $\{\psi(x)\mid x\in V_s(\mathfrak{h})\}$. Note that the isomorphism class of $\Lambda^{\rm{red}}(\mathfrak{h},m)$ depends only on $m$, and is independent of $\psi$, justifying the notation.
\begin{lemma}
\label{reduced representation lemma}
  If $\mathfrak{h}$ is a Hoffman signed graph having a representation of norm $m$, then $\mathfrak{h}$ has a reduced representation of norm $m$, and $\Lambda(\mathfrak{h},m)$ is isomorphic to $\Lambda^{\rm{red}}(\mathfrak{h},m)\oplus\mathbb{Z}^{|V_f(\mathfrak{h})|}$ as a lattice. 
\end{lemma}
\begin{proof}
  Let $\phi:V(\mathfrak{h})\rightarrow\mathbb{R}^n$ be a representation of norm $m$. Let $U$ be the subspace of $\mathbb{R}^n$ generated by $\{\phi(x)\mid x\in V_f(\mathfrak{h})\}$. Since $(\phi(x),\phi(y))=0$ for $x,y\in V_f(\mathfrak{h})$, we obtain $\{\phi(x)\mid x\in V_f(\mathfrak{h})\}$ is an orthogonal basis of $U$. Let $\psi(x)=\phi(x)-\sum_{y\in V_f(\mathfrak{h})}(\phi(x),\phi(y))\phi(y)$ for $x\in V_s(\mathfrak{h})$, then $\psi(x)$ is a reduced representation of norm $m$. Note that $(\psi(x),\phi(y))=0$ for $x\in V_s(\mathfrak{h}),y\in V_f(\mathfrak{h})$. Therefore, the lattice $\Lambda(\mathfrak{h},m)$ is isomorphic to $\Lambda^{\rm{red}}(\mathfrak{h},m)\oplus\mathbb{Z}^{|V_f(\mathfrak{h})|}$.
\end{proof}

For the unsigned case, Lemma \ref{reduced representation lemma} was shown by Jang et al. \cite{Jang2014on}. Our proof of the above lemma is essentially identical to theirs.

\begin{theorem}\label{reduced representation theorem}
  For a Hoffman signed graph $\mathfrak{h}=(H,\sigma,\ell)$ with smallest eigenvalue $\lambda_{\min}(\mathfrak{h})$, the following are equivalent.
  \begin{enumerate}[label=\rm{(\roman*)}]
    \item $\mathfrak{h}$ has a representation of norm $m$.
    \item $\mathfrak{h}$ has a reduced representation of norm $m$.
    \item $\lambda_{\min}(\mathfrak{h})\geq -m$.
  \end{enumerate}
\end{theorem}
\begin{proof}
  It follows from Lemma \ref{reduced representation lemma} that $\rm{(i)}$ implies $\rm{(ii)}$.\par
Let $Sp(\mathfrak{h})$ be the special matrix of $\mathfrak{h}$. Let $\psi$ be a reduced representation of $\mathfrak{h}$ of norm $m$. The matrix $Sp(\mathfrak{h})+m\mathbf{I}$ is the Gram matrix of the image of $\psi$. Hence the matrix $Sp(\mathfrak{h})+m\mathbf{I}$ is positive semidefinite, in other words, the matrix $Sp(\mathfrak{h})$ has smallest eigenvalue at least $-m$. This shows that $\rm{(ii)}$ implies $\rm{(iii)}$.\par
Now we prove that $\rm{(iii)}$ implies $\rm{(i)}$. Assume that $\lambda_{\min}(\mathfrak{h}) \geq -m$, then the matrix $Sp(\mathfrak{h})+m\mathbf{I}$ is positive semidefinite, hence has a Cholesky factorization $LL^T$ with square $L$. The following map
$$
\varphi(x)=\left\{\begin{array}{ll}
\mathbf{e}_x, &\text{ if $x$ is fat,}\\\sum\limits_{y\in V_s(\mathfrak{h})}L_{xy}\mathbf{e}_y+\sum\limits_{y\in V_f(\mathfrak{h}),\sigma(\{x,y\})=+}\mathbf{e}_y-\sum\limits_{y\in V_f(\mathfrak{h}),\sigma(\{x,y\})=-}\mathbf{e}_y,&\text{ otherwise,}
\end{array}\right.$$
is a representation of $\mathfrak{h}$ of norm $m$, where $\{\mathbf{e}_x\mid x\in V(\mathfrak{h})\}$ is a set of standard unit vectors. This completes the proof.
\end{proof}
As a direct consequence of Theorem \ref{reduced representation theorem}, we find the relationship between the smallest eigenvalue of a Hoffman signed graph and the smallest eigenvalue of its induced Hoffman signed subgraph. For the unsigned case, this result was shown by Woo and Neumaier \cite{woo1995on}. 

\begin{lemma}\label{sub-minev}
  Let $\mathfrak{h}$ be a Hoffman signed graph. If $\mathfrak{g}$ is an induced Hoffman signed subgraph of $\mathfrak{h}$, then $\lambda_{\min}(\mathfrak{g})\geq\lambda_{\min}(\mathfrak{h})$. In particular,  $\lambda_{\min}(H_s,\sigma_s)\geq\lambda_{\min}(\mathfrak{h})$, where $(H_s,\sigma_s)$ is the slim graph of $\mathfrak{h}$.
\end{lemma}
\begin{proof}
    Assume that $\lambda_{\min}(\mathfrak{h})=-m$. By Theorem \ref{reduced representation theorem}, $\mathfrak{h}$ has a representation of norm $m$. Since $\mathfrak{g}$ is an induced Hoffman signed subgraph of $\mathfrak{h}$, then $\mathfrak{g}$ also has a representation of norm $m$. Thus, $\lambda_{\min}(\mathfrak{g})\geq -m=\lambda_{\min}(\mathfrak{h})$, by Theorem \ref{reduced representation theorem}. As the slim graph $(H_s,\sigma_s)$ is an induced Hoffman signed subgraph of $\mathfrak{h}$, we obtain that $\lambda_{\min}(H_s,\sigma_s)\geq\lambda_{\min}(\mathfrak{h})$.
\end{proof}

\begin{lemma}\label{integrability of two lattices}
  Let $\mathfrak{h}$ be a Hoffman signed graph with smallest eigenvalue $\lambda_{\min}(\mathfrak{h})$, and let $m$ be an integer with $m\geq -\lambda_{\min}(\mathfrak{h})$. The integral lattice $\Lambda^{\rm{red}}(\mathfrak{h},m)$ is $s$-integrable if and only if the integral lattice $\Lambda(\mathfrak{h},m)$ is $s$-integrable.
\end{lemma}
\begin{proof}
    By Lemma \ref{reduced representation lemma} and the definition of integrability of lattices, we obtain $\Lambda^{\rm{red}}(\mathfrak{h},m)$ is $s$-integrable if and only if the integral lattice $\Lambda(\mathfrak{h},m)$ is $s$-integrable.
\end{proof}

 \begin{corollary}\label{integrability of slim graph}
  Let $\mathfrak{h}$ be a Hoffman signed graph with smallest eigenvalue $\lambda_{\min}(\mathfrak{h})$, and let $(H_s,\sigma_s)$ be the slim graph of $\mathfrak{h}$ with smallest eigenvalue $\lambda_{\min}(H_s,\sigma_s)$. Suppose $\lfloor\lambda_{\min}(\mathfrak{h})\rfloor=\lfloor\lambda_{\min}(H_s,\sigma_s)\rfloor$ holds. The graph $(H_s,\sigma_s)$ is $s$-integrable if the integral lattice $\Lambda(\mathfrak{h}, -\lfloor\lambda_{\min}(\mathfrak{h})\rfloor)$ is $s$-integrable.  
\end{corollary}
\begin{proof}
    This follows from Lemma \ref{integrability of two lattices} immediately. 
\end{proof}

\subsection{Associated Hoffman signed graphs}\label{subsec ahsg}
Let $m\geq 2,n\geq 2(m^2+m)$ be integers and $(G,\sigma)$ be a $\{\widetilde{K}_{2m}^{(0)},\widetilde{K}_{2m}^{(-)}\}$-switching-free signed graph. Let $\mathcal{C}(n)$ denote the set of maximal positive cliques of $(G,\sigma)$ with at least $n$ vertices. For $C\in\mathcal{C}(n)$, recall that $N_{(m)}^{(+)}(V(C))$ (resp. $N_{(m)}^{(-)}(V(C))$) denotes the $m$-positive-neighborhood (resp. the $m$-negative-neighborhood) of $C$. For convenience, let $N_{(m)}(V(C)):=N_{(m)}^{(+)}(V(C))\cup N_{(m)}^{(-)}(V(C))$. In Subsection \ref{subsec qpc}, we have seen that if switching the subgraph induced on $N_{(m)}(V(C))$ with respect to $N^{(-)}_{(m)}(V(C))$, we obtain the quasi-positive-clique with respect to $C$. Now we are ready to define associated Hoffman signed graphs.

\begin{definition}
Let $m\geq 2,n\geq 2(m^2+m)$ be integers and $(G,\sigma)$ be a $\{\widetilde{K}_{2m}^{(0)},\widetilde{K}_{2m}^{(-)}\}$-switching-free signed graph. Let $\mathcal{C}(n)$ denote the set of maximal positive cliques of $(G,\sigma)$ with at least $n$ vertices. Let $\{N_{(m)}(V(C))\mid C\in\mathcal{C}(n)\}=\{N_{(m)}(V(C_1)),\ldots, N_{(m)}(V(C_r))\}$. The Hoffman signed graph satisfying the following conditions is an \emph{associated Hoffman signed graph} of $(G,\sigma)$, with respect to the pair $(m,n)$, denoted by $\mathfrak{g}=\mathfrak{g}(G,\sigma,m,n)$.
\begin{enumerate}[label=\rm{(\roman*)}]
    \item $V_s(\mathfrak{g})=V(G)$, and the slim graph of $\mathfrak{g}$ is $(G,\sigma)$.
    \item $V_f(\mathfrak{g})=\{F_1,F_2,\ldots,F_r\}$, and each fat vertex $F_i$ is positively adjacent to each vertex in $N_{(m)}^{(+)}(V(C_i))$, negatively adjacent to each vertex in $N_{(m)}^{(-)}(V(C_i))$, and non-adjacent to the other vertices of $V(G)$, for $i=1,\ldots,r$.
\end{enumerate}
\end{definition}

Remark \ref{remark ahsg} is important to help us understand associated Hoffman signed graphs better.

\begin{remark}\label{remark ahsg}
\begin{enumerate}[label=\rm{(\roman*)}]
    \item In the above definition, the set $\{C_1,\ldots,C_r\}$ may merely be a proper subset of $\mathcal{C}(n)$, as there may be two distinct positive cliques $C,C'\in\mathcal{C}(n)$, such that $N_{(m)}(V(C))=N_{(m)}(V(C'))$, by Lemma \ref{two cliques signed case}.
    \item  It is possible that there exists another subset $\{C'_1,\ldots,C'_r\}$ of $\mathcal{C}(n)$, such that $N_{(m)}(V(C_1))=N_{(m)}(V(C'_1))$ for $i=1,\ldots,r$, then we obtain another associated Hoffman signed graph $\mathfrak{g}'$ of $(G,\sigma)$. Note that $\mathfrak{g'}$ may be different from $\mathfrak{g}$. However, they have the same special matrix.
\end{enumerate}
\end{remark}

Next proposition connects signed graphs and Hoffman signed graphs, which is crucial in the next section.

\begin{proposition}{\label{associatedProp}}
Let $m\geq 2,v_f\geq0,v_s\geq0,p\geq0$ be integers. There exists a positive integer $q=q(m,v_f,v_s,p)\geq 2(m^2+m)$, such that, for any signed graph $(G,\sigma)$, any integer $n\geq q$ and any Hoffman signed graph $\mathfrak{h}$ with at most $v_f$ fat vertices and at most $v_s$ slim vertices, the signed graph $G(\mathfrak{h},p)$ is isomorphic to an induced subgraph of $(G,\sigma)$, provided that the following hold.
  \begin{enumerate}[label=\rm{(\roman*)}]
    \item The signed graph $(G,\sigma)$  is $\{\widetilde{K}_{2m}^{(0)},\widetilde{K}_{2m}^{(-)}\}$-switching-free.
    \item The Hoffman signed graph $\mathfrak{h}$ is isomorphic to an induced subgraph of an associated Hoffman signed graph $\mathfrak{g}(G,\sigma,m,n)$.
  \end{enumerate}
\end{proposition}
\begin{proof} Let $q=\max_{0\leq v_f'\leq v_f,0\leq v_s'\leq v_s}\{q'(m,v_f',v_s',p)\}$, where $q'(m,0,v_s',p):=2(m^2+m)$, and $q'(m,v_f',v_s',p):=\max\{q'(m,v_f'-1,v_s',p+2m-2),(2m-1)v_s'+(2m-2)(v_f'-1)p+p\}$ for $v_f'\geq1$. Let $n\geq q$, $(G,\sigma)$ be a signed graph satisfying the given conditions. Let $\mathfrak{g}:=\mathfrak{g}(G,\sigma,m,n)$ be an associated Hoffman signed graph of $(G,\sigma)$. Let $\mathfrak{h}$ be an induced Hoffman subgraph of $\mathfrak{g}$ with $v_f'$ fat vertices and $v_s'$ slim vertices where $v_f'\leq v_f$ and $v_s'\leq v_s$.
We will prove this proposition by induction on $v_f'$.\par

When $v_f'=0$, $\mathfrak{h}$ contains no fat vertices. Thus, $G(\mathfrak{h},p)$ is the slim graph of $\mathfrak{{h}}$. It is clear that $G(\mathfrak{h},p)$ is an induced subgraph of $(G,\sigma)$.\par

Now we consider the case $v_f'\geq 1$. Let $V_f(\mathfrak{h})=\{F_1,F_2,\ldots,F_{v_f'}\}$ be the fat vertex set of $\mathfrak{h}$ and $V_s(\mathfrak{h})$ be the slim vertex set of $\mathfrak{h}$. By the definition of associated Hoffman signed graphs, there exists $C_i\in\mathcal{C}(n)$, such that  the set $N_{(m)}(V(C_i))$ is the neighborhood of $F_i$ in $\mathfrak{g}$, for $i=1,2,\ldots,v_f'$.\par

Let $\mathfrak{h}_1$ denote the subgraph of $\mathfrak{h}$ obtained by deleting the fat vertex $F_{v_f'}$. By the induction hypothesis, the signed graph $G(\mathfrak{h}_1,p+2m-2)$ is isomorphic to an induced subgraph of $(G,\sigma)$. Denote by $D_i$ the positive clique in $G(\mathfrak{h}_1,p+2m-2)$ replacing the fat vertex $F_i$ for $i=1,2,\ldots,v_f'-1$.\par
By Lemma \ref{neighborhood of 3m-2-clique signed case}, for $x\in V_s(\mathfrak{h})$, if $x$ is positively adjacent to $F_{v_f'}$, then $x$ has at most $2m-2$ non-positive neighbors in $C_{v_f'}$; if $x$ is negatively adjacent to $F_{v_f'}$, then $x$ has at most $2m-2$ non-negative neighbors in $C_{v_f'}$; if $x$ is not adjacent to $F$, then it has at most $2m-2$ neighbors in $C_{v_f'}$. Therefore, we are able to obtain a positive clique $D_{v_f'}$ containing at least $n-v_s'-(2m-2)v_s'$ vertices, such that $V_s(\mathfrak{h})\cap V(D_{v_f'})=\emptyset$, and for $x\in  V_s(\mathfrak{h})$, if $x$ is positively adjacent to $F_{v_f'}$, then $x$ is positively adjacent to each vertex of $D_{v_f'}$; if $x$ is negatively adjacent to $F_{v_f'}$, then $x$ is negatively adjacent to each vertex of $D_{v_f'}$; if $x$ is non-adjacent to $F_{v_f'}$, then $x$ is non-adjacent to each vertex of $D_{v_f'}$.\par

By Lemma \ref{two cliques signed case}, $|V(D_i)-N_{(m)}(V(C_{v_f'}))|\geq p+2m-2-2(m-1)=p$, thus there exists a positive clique of order $p$ with vertex set in $V(D_i)-N_{(m)}(V(C_{v_f'}))$, denoted by $D_i'$. Also, $|V(D_{v_f'})-\{x\in V(D_{v_f'})\mid x \text{ has at least one neighbor in }\bigcup_{i=1}^{v_f'-1}V(D_i')\}|\geq |V(D_{v_f'})|-\sum_{i=1}^{v_f'-1}(2m-2)|V(D_i')|\geq n-v_s'-(2m-2)v_s'-(v_f'-1)(2m-2)p\geq p$. Hence, there exists a positive clique of order $p$ in $D_{v_f'}$, denoted by $D_{v_f'}'$, such that each vertex in $D_{v_f'}'$ is not adjacent to any vertex in $\bigcup_{i=1}^{v_f'-1}V(D_i')$.\par
Note that $D_1',D_2',\ldots, D_{v_f'-1}'$ are isolated positive cliques in $(G,\sigma)$ with $p$ vertices. Thus, the subgraph of $(G,\sigma)$ induced on $V_s(\mathfrak{h})\cup(\bigcup^{v_f'}_{i=1}V(D_i'))$ is isomorphic to $G(\mathfrak{h},p)$.\par
\end{proof} 

\section{\texorpdfstring{Proofs of Theorem \ref{equaity of fobidden subgraphs and min ev} $\rm{(ii)}$ and Theorem \ref{structure theory thm}}{Proof of Theorem \ref{equaity of fobidden subgraphs and min ev} (ii) and Theorem \ref{structure theory thm}}}\label{sec of pf thm}
 In this section, we will prove Theorem \ref{equaity of fobidden subgraphs and min ev} $\rm{(ii)}$ and Theorem \ref{structure theory thm}. First, we show the following theorem.
\begin{theorem}\label{structure with forbidden graphs}
  Let $m\geq 2,s\geq 3,t\geq 2$ be integers. There exist two positive integers $        q=q(m,t),\kappa_4=\kappa_4(m,s,t)$ such that for any $\{\widetilde{K}_{2m}^{(0)},\widetilde{K}_{2m}^{(-)}$, $(K_s,-)$, $(K_{1,t},+)\}$-switching-free signed graph $(G,\sigma)$ and any integer $n\geq q$, if $\mathfrak{g}=\mathfrak{g}(G,\sigma,m,n)$ is an associated Hoffman signed graph  of $(G,\sigma)$ with respect to the pair $(m,n)$, then the subgraphs $N_1,\ldots,N_r$ $(r=|V_f(\mathfrak{g})|)$ induced on the neighbors of each fat vertex of $\mathfrak{g}$ satisfy the following.\par
  \begin{enumerate}[label=\rm{(\roman*)}]
    \item Each vertex in $(G,\sigma)$ is contained in at most $t-1$ $N_i$'s.
    \item The induced subgraph $N_i$ is switching equivalent to a signed graph whose positive graph is a $\kappa_4$-plex, for $i=1,\ldots,r$.
    \item The intersection $V(N_i)\cap V(N_j)$ contains at most $\kappa_4$ vertices, for $1\leq i<j\leq r$.
    \item The subgraph $(G',\sigma')$ has maximum valency at most $R(n,s,t)-1$, where $G'= (V(G),E(G)\backslash\bigcup_{i=1}^r E(N_i))$ and $\sigma'=\sigma|_{E(G')}$.
  \end{enumerate}
\end{theorem}

\begin{proof}
Let $q:=\max\{2(m^2+m),4(t-1)(m-1)+m\}$ and $\kappa_4:=\max\{\kappa_2,4\kappa_3\}$, where $\kappa_2$ and $\kappa_3$ are as defined in Proposition \ref{quasi-cli sw to almost posi-clique} and Lemma \ref{quasi-cli almost isolate} respectively. For any $n\geq q$, denote by $\mathcal{C}(n)$ the set $\{C\mid C$ is a maximal positive clique of $(G,\sigma)$ with at least $n$ vertices$\}$, and for $C\in\mathcal{C}(n)$, let $N_{(m)}(V(C))=N_{(m)}^{(+)}(V(C))\cup N_{(m)}^{(-)}(V(C))$. Let $\{N_{(m)}(V(C_1)),\ldots,N_{(m)}(V(C_r))\}=\{N_{(m)}(V(C))\mid C\in \mathcal{C}(n)\}$. Let $N_i:=(G,\sigma)_{N_{(m)}(V(C_i))}$ denote the subgraph of $(G,\sigma)$ induced on $N_{(m)}(V(C_i))$ for $i=1,\ldots,r$. \par

$\rm{(i)}$ Suppose that there exists a vertex $x$ such that $x$ is contained in at least $t$ $N_i$'s. Without loss of generality, we may assume that $x\in \bigcap_{i=1}^t N_{(m)}(V(C_i))$. Let $U_i=V(C_i)\cap\{y\mid y\text{ is adjacent to }x\}-\bigcup_{1\leq j\leq t}^{j\neq i}N_{(m)}(V(C_j))$. By Lemma \ref{neighborhood of 3m-2-clique signed case} and Lemma \ref{two cliques signed case}, $|U_i|\geq n-(m-1)-(t-1)(2m-2)\geq 4(t-1)(m-1)+m-(m-1)-(t-1)(2m-2)\geq (t-1)(2m-2)+1$, for $1\leq i\leq t$. Take a vertex $y_1$ from $U_1$. For $i=2,\ldots,t$, there exists $y_i\in U_i-\bigcup_{j=1}^{i-1}\{y\mid y\text{ is adjacent to }y_j\}$, since $|U_i-\bigcup_{j=1}^{i-1}\{y\mid y\text{ is adjacent to }y_j\}|\geq(t-1)(2m-2)+1-(i-1)(2m-2)\geq1$. The subgraph induced on $\{x,y_1,y_2,\ldots,y_t\}$ is switching equivalent to $(K_{1,t},+)$, which is a contradiction.\par

$\rm{(ii)}$ Let $Q_i=Q(C_i)$ be the quasi-positive-clique obtained from $N_i$ by switching on $N_{(m)}^{(-)}(V(C_i))$. By Proposition \ref{quasi-cli sw to almost posi-clique}, the positive graph of $Q_i$ is a $\kappa_4$-plex, as $\kappa_4\geq\kappa_2$. Therefore, (ii) holds.\par

$\rm{(iii)}$ By Lemma \ref{quasi-cli almost isolate} , we have $|V(N_i)\cap V(N_j)|\leq \sum_{\varepsilon_1,\varepsilon_2\in \{+,-\}}|V(N_{(m)}^{(\varepsilon_1)}(V(C_i))\cap V(N_{(m)}^{(\varepsilon_2)}(V(C_j))|\leq 4\kappa_3\leq \kappa_4$ for $1\leq i<j\leq r$.  \par

$\rm{(iv)}$ Suppose that there exists a vertex $z$ in $(G',\sigma')$ with valency at least $R(n,s,t)$. Let $(H,\sigma_H)$ denote the subgraph of $(G,\sigma)$ induced on the neighbors of $z$ in $(G',\sigma')$. Since $(G,\sigma)$ is $\{(K_s,-)$, $(K_{1,t},+)\}$-switching-free, the subgraph $(H,\sigma_H)$ contains a $(K_n,+)$, by Theorem \ref{Ramseynumber}. Notice that $(K_n,+)$ is contained in some maximal positive clique $C\in\mathcal{C}(n)$ in $(G,\sigma)$. Since $\{N_{(m)}(V(C_1)),\ldots,N_{(m)}(V(C_r))\}=\{N_{(m)}(V(C))\mid C\in \mathcal{C}(n)\}$, there exists $i$ such that $N_{(m)}(C)=N_{(m)}(C_i)$. Note that $z$ is adjacent to each vertex in this $(K_n,+)$. Thus, by Lemma \ref{neighborhood of 3m-2-clique signed case}, $z\in N_{(m)}(C)=N_{(m)}(C_i)$, and the edges between $z$ and this $(K_n,+)$ are in $N_i$, not in $(G',\sigma')$. This gives a contradiction and $\rm{(iv)}$ holds.
\end{proof}
Now we prove Theorem \ref{equaity of fobidden subgraphs and min ev} $\rm{(ii)}$.\par
\noindent
\textbf{Proof of Theorem \ref{equaity of fobidden subgraphs and min ev}} $\rm{(ii)}$. When $t=1$, $(G,\sigma)$ is an empty graph, as it is $(K_2,-)$-switching-free. In this case, $\lambda_{\min}(G,\sigma)=0$. Now we may assume $t\geq 2$. Let $\kappa:=\max\{\kappa_4(t,t+1,t),R(q(t,t),t+1,t)-1\}$, where $q(t,t),\kappa_4(t,t+1,t)$ are as defined in Theorem \ref{structure with forbidden graphs}. Let $\mathfrak{g}:=\mathfrak{g}(G,\sigma,t,q(t,t))$ be an associated Hoffman graph of $(G,\sigma)$ with respect to the pair $(t,q(t,t))$ and $N_1,N_2,\ldots,N_r$ be the subgraphs induced on the neighbors of each fat vertex of $\mathfrak{g}$, where $r=|V_f(\mathfrak{g})|$. Let $S:=S(\mathfrak{g})$ be the special matrix of $\mathfrak{g}$. 
By Lemma \ref{sub-minev}, we have $\lambda_{\min}(G,\sigma)\geq\lambda_{\min}(\mathfrak{g})$.\par
Given a vertex $x$ of $(G,\sigma)$, assume that $N_1,\ldots,N_p$ are all subgraphs among $N_1,\ldots,N_r$ which contain $x$. By Theorem \ref{structure with forbidden graphs}, $p\leq t-1$. Now we are going to look at the vertices $y$ of $(G,\sigma)$ such that $S_{xy}\neq 0$. 
Let 
\begin{align*}
    &U=\{y\mid y\in V_s(\mathfrak{g})-\{x\},S_{xy}\neq0\},  \\
     &U_0=\{y\in U\mid y \text{ and }x \text{ have no common fat neighbors}\}, \\
     &U_1=\{y\in U\mid y \text{ and }x \text{ have exactly one common fat neighbor}\}, \\
     &U_{\geq 2}=\{y\in U\mid y \text{ and }x \text{ have at least two common fat neighbors}\}.
\end{align*}

For any $y\in U_0$, $|S_{xy}|=1$ holds, and by Theorem \ref{structure with forbidden graphs} $\rm{(iv)}$, we have $|U_0|\leq R(q(t,t),t+1,t)-1\leq \kappa$.
For any $y\in U_1$, $|S_{xy}|\leq2$ holds. Note that $|U_1|=|U_1\cap(\bigcup_{i=1}^p V(N_i))|\leq \sum_{i=1}^p|U_i\cap V(N_i)|$. Let $Q_i$ denote the corresponding quasi-positive-clique with respect to $N_i$. For any $y\in U_1\cap V(N_i)$, $S_{xy}\neq0$ implies that $x$ and $y$ are not positively adjacent in $Q_i$. By Theorem \ref{structure with forbidden graphs} $\rm{(ii)}$, we have $|U_1\cap V(N_1)|\leq \kappa_4-1\leq \kappa-1$. Hence, $|U_1|\leq p(\kappa-1)\leq (t-1)(\kappa-1)$.
For any $y\in U_{\geq2}$, since $x$ has at most $t-1$ fat neighbors, $|S_{xy}|\leq 1+(t-1)=t$ holds. Moreover, if $y\in U_{\geq2}$, then $x,y\in V(N_i)\cap V(N_j)$ for some $1\leq i<j \leq p$. By Theorem \ref{structure with forbidden graphs} $\rm{(iii)}$, $|V(N_i)\cap V(N_j)|\leq \kappa_4\leq \kappa$. Thus, $|U_{\geq2}|\leq \binom{p-1}{2}\kappa_4\leq\binom{t-1}{2}\kappa$.\par

Therefore, 
\begin{align*}
    \sum_y|S_{xy}|&=|S_{xx}|+\sum_{y\in U_0}|S_{xy}|+\sum_{y\in U_1}|S_{xy}|+\sum_{y\in U_{\geq2}}|S_{xy}|\\
    &\leq t-1+|U_0|+2|U_1|+t\cdot|U_{\geq2}|\\
    &\leq t-1+\kappa+2(t-1)(\kappa-1)+t\binom{t-1}{2}\kappa \\
    &=\frac{t(t-1)(t-2)\kappa}{2}+2t\kappa-\kappa-t+1,
\end{align*} 
and 
$$\lambda_{\min}(G,\sigma)\geq \lambda_{\min}(\mathfrak{g})\geq -\max_x\left|{\sum_yS_{xy}}\right|\geq -\max_x{\sum_y|S_{xy}|}\geq-\frac{t(t-1)(t-2)\kappa}{2}-2t\kappa+\kappa+t-1.$$
This completes the proof.
\qed

For signed graph with fixed smallest eigenvalue, we obtain the following result.

\begin{theorem}\label{structure th for fixed min ev}
Let $\lambda< -1$ be a real number. There exist positive integers $m_\lambda,q_\lambda$ such that for any integer $n\geq q_{\lambda}$, a positive integer $d=d(\lambda,n)$ satisfying the following exists.\par
For any signed graph $(G,\sigma)$ with smallest eigenvalue at least $\lambda$ and minimum valency at least $d(\lambda,n)$, if $\mathfrak{g}(G,\sigma,m_\lambda,n)$ is an associated Hoffman signed graph of $(G,\sigma)$ with respect to the pair $(m_\lambda,n)$, then the subgraphs $N_1,N_2,\ldots,N_r$ $(r=|V_f(\mathfrak{g})|)$ induced on the neighbors of each fat vertex of $\mathfrak{g}$ have the following properties. 
\begin{enumerate}[label=\rm{(\roman*)}]
  \item Each vertex in $(G,\sigma)$ is contained in at least one and at most $\lfloor-\lambda\rfloor$ $N_i$'s.
  \item The induced subgraph $N_i$ is switching equivalent to a signed graph whose positive graph is a $\lfloor \lambda^2+2\lambda+2\rfloor$-plex, for $i=1,2,\ldots,r$.
  \item The intersection $V(N_i)\cap V(N_j)$ contains at most $4\lfloor-\lambda\rfloor-4$ vertices for $1\leq i< j\leq r$.
  \item The subgraph $(G',\sigma')$ has maximum valency at most $d(\lambda,n)-1$, where $G'= (V(G),E(G)\backslash\bigcup_{i=1}^r E(N_i))$ and $\sigma'=\sigma|_{E(G')}$.
\end{enumerate}
\end{theorem}
\begin{proof}
Considering $\lim\limits_{m\rightarrow +\infty}\lambda_{\min}(\widetilde{K}_{2m}^{(0)})=\lim\limits_{m\rightarrow +\infty}\lambda_{\min}(\widetilde{K}_{2m}^{(-)})=-\infty$, let $m_{\lambda}=\min\{m\mid \lambda_{\min}(\widetilde{K}_{2m}^{(0)})<\lambda$ and $\lambda_{\min}(\widetilde{K}_{2m}^{(-)})<\lambda\}$. To obtain $q_{\lambda}$, we need to introduce a few families of Hoffman signed graphs with smallest eigenvalue less than $\lambda$. \par
Let $\mathfrak{h}_0$ be the fat Hoffman signed graph with $|V_s(\mathfrak{h}_0)|=1$ and $|V_f(\mathfrak{h}_0)|=\lfloor-\lambda\rfloor+1$, such that the unique slim vertex is positively adjacent to all the fat vertices of $\mathfrak{h}_0$. Let $\mathfrak{H}_1$ be the set of fat Hoffman signed graphs with exactly $\lfloor(\lambda+1)^2\rfloor+1$ slim vertices and one fat vertex, such that the unique fat vertex is positively adjacent to all slim vertices and there exists at least one slim vertex which is non-positively adjacent to all of the other slim vertices. Let $\mathfrak{H}_2$ be the set of fat Hoffman signed graphs with exactly $\lfloor-\lambda\rfloor$ slim vertices and two fat vertices, such that each fat vertex is positively adjacent to all of the slim vertices. Let $\Sigma_0=\{\mathfrak{h}\mid \mathfrak{h}$ is switching equivalent to $\mathfrak{h}_0\}$, $\Sigma_i=\{\mathfrak{h}\mid \mathfrak{h}$ is switching equivalent to some member in $\mathfrak{H}_i\}$ for $i=1,2$.
Now we show that for any $\mathfrak{h}\in\Sigma_0\cup\Sigma_1\cup\Sigma_2$, $\lambda_{\min}(\mathfrak{h})<\lambda$ holds. For any $\mathfrak{h}\in\Sigma_0$, we have $\lambda_{\min}(\mathfrak{h})=\lambda_{\min}(\mathfrak{h}_0)\leq-(\lfloor-\lambda\rfloor+1)<\lambda$, as $\mathfrak{h}$ and $\mathfrak{h}_0$ are switching equivalent, by Lemma \ref{hsg sw}.
For any $\mathfrak{h}\in\Sigma_1$, assume that $\mathfrak{h}$ is switching equivalent to $\mathfrak{h}_1\in\mathfrak{H}_1$, where $x$ is a slim vertex of $\mathfrak{h}_1$ which is non-positively adjacent to all of the other slim vertices. It follows that 
$$
S(\mathfrak{h}_1)_{xy}=\left\{\begin{array}{ll}
-1, &\text{ if $x$ and $y$ are non-adjacent,}\\-2,&\text{ otherwise,}
\end{array}\right.$$
for $y\in V_s(\mathfrak{h}_1)$. Note that $S(\mathfrak{h}_1)_{yy}=-1$ and $S(\mathfrak{h}_1)_{yy'}\leq0$ for any $y,y'\in V_s(\mathfrak{h}_1)$, as $\mathfrak{h}_1$ has a unique fat vertex positively adjacent to all the slim vertices of $\mathfrak{h}_1$. By Perron-Frobenius Theorem, we have $\lambda_{\max}(-S(\mathfrak{h}_1))\geq\lambda_{\max}(A(K_{1,\lfloor(\lambda+1)^2\rfloor+1},+)+\mathbf{I})=\sqrt{\lfloor(\lambda+1)^2\rfloor+1}+1> -\lambda$. Thus, by Lemma \ref{hsg sw}, $\lambda_{\min}(\mathfrak{h})=\lambda_{\min}(\mathfrak{h}_1)=-\lambda_{\max}(-S(\mathfrak{h}_1))<\lambda$.
For any $\mathfrak{h}\in\Sigma_2$, assume that $\mathfrak{h}$ is switching equivalent to $\mathfrak{h}_2\in\mathfrak{H}_2$. Since $\mathfrak{h}_2$ has exactly two fat vertices and each slim vertex is positively adjacent to both of the fat vertices, the special matrix $S(\mathfrak{h}_2)$ satisfies that 
$$
S(\mathfrak{h}_2)_{xy}=\left\{\begin{array}{ll}
-2, &\text{ if $x=y$,}\\
-1,&\text{ if $x$ and $y$ are positively adjacent,}\\
-2,&\text{ if $x$ and $y$ are non-adjacent,}\\
-3,&\text{ if $x$ and $y$ are negatively adjacent.}
\end{array}\right.$$
Applying the Perron-Frobenius Theorem again yields $\lambda_{\max}(-S(\mathfrak{h}_2))\geq\lambda_{\max}(A(K_{\lfloor-\lambda\rfloor},+)+2\mathbf{I})\geq \lfloor-\lambda\rfloor+1>-\lambda$, then by Lemma \ref{hsg sw} again, 
$\lambda_{\min}(\mathfrak{h})=\lambda_{\min}(\mathfrak{h}_2)=-\lambda_{\max}(-S(\mathfrak{h}_2))<\lambda$. 
By Theorem \ref{Hoffman-Ostrowski Theorem of signed case}, for any $\mathfrak{h}\in \Sigma_0\cup\Sigma_1\cup\Sigma_2$, there exists a positive integer $p_{\mathfrak{h}}$ such that $\lambda_{\min}(G(\mathfrak{h},p_{\mathfrak{h}}))<\lambda$. Let  $v_f=\max\{|V_f(\mathfrak{h})|\mid \mathfrak{h}\in \Sigma_0\cup\Sigma_1\cup\Sigma_2\}$, $v_s=\max\{|V_s(\mathfrak{h})|\mid \mathfrak{h}\in \Sigma_0\cup\Sigma_1\cup\Sigma_2\}$ and $p=\max\{p_{\mathfrak{h}}\mid \mathfrak{h}\in \Sigma_0\cup\Sigma_1\cup\Sigma_2\} $. Let $q=q(m_{\lambda},v_f,v_s,p)$ be the integer as defined in Proposition \ref{associatedProp}. Let $q_{\lambda}=\max\{2(m_{\lambda}^2+m_{\lambda}),(2m_{\lambda}-1)t+1,q\}$. \par

For any integer $n\geq q_{\lambda}$, let $d(\lambda,n)=R(n,2-\lfloor\lambda\rfloor,\lfloor\lambda-1\rfloor^2)$. Let $\mathfrak{g}:=\mathfrak{g}(G,\sigma,m_{\lambda},n)$ be an associated Hoffman graph of $(G,\sigma)$ with respect to the pair $(m_{\lambda},n)$ and $N_1,N_2,\ldots,N_r$ be the subgraphs induced on the neighbors of each fat vertex of $\mathfrak{g}$, where $r=|V_f(\mathfrak{g})|$. Since $\lambda_{\min}(G,\sigma)\geq \lambda$, $(G,\sigma)$ is $\{\widetilde{K}_{2m_{\lambda}}^{(0)},\widetilde{K}_{2m_{\lambda}}^{(-)}\}$-switching-free and $\max\{\lambda_{\min}(\widetilde{K}_{2m_{\lambda}}^{(0)}),\lambda_{\min}(\widetilde{K}_{2m_{\lambda}}^{(-)})\}<\lambda$ and $\lambda_{\min}(G(\mathfrak{h},p))\leq\lambda_{\min}(G(\mathfrak{h},p_{\mathfrak{h}}))<\lambda$, then $(G,\sigma)$ does not contain $G(\mathfrak{h},p)$ as an induced subgraph for any $\mathfrak{h}\in \Sigma_0\cup\Sigma_1\cup\Sigma_2$. 
Thus, $\mathfrak{g}$ does not contain any Hoffman signed graph in $\Sigma_0\cup\Sigma_1\cup\Sigma_2$ as an induced subgraph, by Proposition \ref{associatedProp}.  \par

$\rm{(i)}$ For an arbitrary vertex $z_0$ of $(G,\sigma)$, let $(H_1,\sigma_{H_1})$ denote the subgraph of $(G,\sigma)$ induced on the neighbors of $z_0$ in $(G,\sigma)$. Note that $(G,\sigma)$ is $\{(K_{2-\lfloor\lambda\rfloor},-),(K_{1,\lfloor\lambda-1\rfloor^2},+)\}$-switching-free, as the smallest eigenvalues of $(K_{2-\lfloor\lambda\rfloor},-)$ and $(K_{1,\lfloor\lambda-1\rfloor^2},+)$ are less than $\lambda$. Since the valency of $z_0$ is at least $R(n,2-\lfloor\lambda\rfloor,\lfloor\lambda-1\rfloor^2)$, by Theorem \ref{Ramseynumber}, the subgraph $(H_1,\sigma_{H_1})$ contains a $(K_n,+)$. 
Notice that this $(K_n,+)$ is contained in some maximal positive clique $C\in\mathcal{C}(n)$ in $(G,\sigma)$. Since $\{N_{(m)}(V(C_1)),\ldots,N_{(m)}(V(C_r))\}=\{N_{(m)}(V(C))\mid C\in \mathcal{C}(n)\}$, there exists $i$ such that $z_0\in N_{(m)}(V(C))=V(N_i)$. Suppose that $z_0$ is contained in at least $\lfloor -\lambda\rfloor+1$ $N_i$'s, then $z_0$ has at least $\lfloor -\lambda\rfloor+1$ fat neighbors in $\mathfrak{g}$. Thus, the subgraph induced on $z_0$ and its $\lfloor -\lambda\rfloor+1$ fat neighbors is contained in $\Sigma_0$, which gives a contradiction. This shows that each vertex of $(G,\sigma)$ is contained in at most $\lfloor -\lambda\rfloor$ $N_i$'s.\par

$\rm{(ii)}$ Let $Q_i=Q(C_i)$ be the quasi-positive-clique obtained from $N_i$ by switching on $N_{(m)}^{(-)}(V(C_i))$ and $F_i$ be the corresponding fat vertex of $N_i$. Suppose that there exists a vertex $z_1$ in $Q_i$ such that $z_1$ has at least $\lfloor(\lambda+1)^2\rfloor+1$ non-positive neighbors in $Q_i$. Let $U$ be a set of non-positive neighbors of $z_1$ in $Q_i$ such that $|U|=\lfloor(\lambda+1)^2\rfloor+1$. Thus, the subgraph of $\mathfrak{g}$ induced on $\{z_1,F_i\}\cup U$ is contained in $\Sigma_1$. This contradiction proves $\rm{(ii)}$.\par

$\rm{(iii)}$ Suppose that there exists a pair $i\neq j\in \{1,\ldots,r\}$ such that $|V(N_i)\cap V(N_j)|\geq 4\lfloor-\lambda\rfloor-3$, then there exists a pair of $\varepsilon_i,\varepsilon_j\in\{+,-\}$ such that $|N_{(m)}^{(\varepsilon_i)}(V(C_i))\cap N_{(m)}^{(\varepsilon_j)}(V(C_j))|\geq \lceil\frac{4\lfloor-\lambda\rfloor-3}{4}\rceil=\lfloor-\lambda\rfloor$. Note that the subgraph of $\mathfrak{g}$ induced on $(N_{(m)}^{(\varepsilon_i)}(V(C_i))\cap N_{(m)}^{(\varepsilon_j)}(V(C_j))\cup \{F_i,F_j\}$ is contained in $\Sigma_2$, which gives a contradiction. Thus, there are at most $4\lfloor-\lambda\rfloor-4$ common vertices in both $N_i$ and $N_j$, for any $1\leq i< j\leq r$.\par

$\rm{(iv)}$ Suppose that there exists a vertex $z_2$ in $(G',\sigma')$ with valency at least $d(\lambda,n)\geq R(n,2-\lfloor\lambda\rfloor,\lfloor\lambda-1\rfloor^2)$. Let $(H_2,\sigma_{H_2})$ denote the subgraph of $(G,\sigma)$ induced on the neighbors of $z_2$ in $(G',\sigma')$. Since $(G,\sigma)$ is $\{(K_{2-\lfloor\lambda\rfloor},-),(K_{1,\lfloor\lambda-1\rfloor^2},+)\}$-switching-free, the subgraph $(H_2,\sigma_{H_2})$ contains a $(K_n,+)$, by Theorem \ref{Ramseynumber}. Notice that this $(K_n,+)$ is contained in some maximal positive clique $C'\in\mathcal{C}(n)$ in $(G,\sigma)$. Since $\{N_{(m)}(V(C_1)),\ldots,N_{(m)}(V(C_r))\}=\{N_{(m)}(V(C))\mid C\in \mathcal{C}(n)\}$, there exists $j$ such that $N_{(m)}(C')=N_{(m)}(C_j)$. Note that $z_2$ is adjacent to each vertex in this $(K_n,+)$. Thus, by Lemma \ref{neighborhood of 3m-2-clique signed case}, $z_2\in N_{(m)}(C')=N_{(m)}(C_j)$, and the edges between $z_2$ and this $(K_n,+)$ are in $N_j$, not in $(G',\sigma')$. This gives a contradiction and $\rm{(iv)}$ holds.\par

This completes the proof.
\end{proof}

Now we prove Theorem \ref{structure theory thm}.\par
\noindent
\textbf{Proof of Theorem \ref{structure theory thm}}. 
Note that for $\lambda<-1$, Theorem \ref{structure theory thm} is a consequence of Theorem \ref{structure th for fixed min ev}.\par
Now we may assume that $\lambda=-1$. 
Note that $\lambda_{\min}(K_3,-)=-2$ and $\lambda_{\min}(K_{1,2},+)=-\sqrt{2}$, then $(G,\sigma)$ is $\{(K_3,-),(K_{1,2},+)\}$-switching-free. This means that $(G,\sigma)$ is switching equivalent to a signed graph where each component is a positive clique. 
Let $d_{\lambda}=1$ and $N_1,\ldots,N_r$ be all the distinct components of $(G,\sigma)$, where $r$ is a positive integer. Note that $N_i$ is switching equivalent to a positive clique $(K_{|N_i|},+)$ for $i=1,\ldots,r$. Hence, the four statements hold.\par
This completes the proof. \qed

\section{\texorpdfstring{Signed graphs with smallest eigenvalue greater than $-1-\sqrt{2}$}{Signed graphs with smallest eigenvalue greater than -1-sqrt{2}}}\label{sec of -1-sqrt2}

In this section, we focus on signed graphs with smallest eigenvalue greater than $-1-\sqrt{2}$ and give a proof of Theorem \ref{thm for min ev >-1-sqrt{2}}.\par

Let $\lambda\leq-1$ be a real number, a \emph{minimal forbidden fat Hoffman signed graph} $\mathfrak{f}$ for $\lambda$ is a fat Hoffman signed graph with smallest eigenvalue less than $\lambda$ such that any proper fat induced Hoffman signed subgraph of $\mathfrak{f}$ has smallest eigenvalue at least $\lambda$. Denote by $\mathcal{F}(\lambda)$ the set of pairwise non-isomorphic minimal forbidden fat Hoffman signed graphs for $\lambda$.\par
We will start with a result on minimal forbidden fat Hoffman signed graphs, which will be used later in the proof of Theorem \ref{thm for min ev >-1-sqrt{2}}. To prove it, we need the following lemma, which was shown by Gavrilyuk et al. \cite{Gavrilyuk2021signed}.

\begin{lemma}[{\cite[Proposition $4.7$]{Gavrilyuk2021signed}}]\label{GMST -sqrt{2}}
Each connected signed graph with smallest eigenvalue greater than $-\sqrt{2}$ is switching equivalent to a positive clique.
\end{lemma}

The following theorem shows that the set $\mathcal{F}(-2)$ is finite. 
 
\begin{theorem}\label{min forbidden finite for -2}
Let $\mathfrak{f}\in\mathcal{F}(-2)$ and $S(\mathfrak{f})$ be the special matrix of $\mathfrak{f}$.
\begin{enumerate}[label=\rm{(\roman*)}]
    \item If $\mathfrak{f}$ contains a slim vertex with at least $2$ fat neighbors, then  $S(\mathfrak{f})$ is
$$(-3)\text{ or }\begin{pmatrix} -2 & a_2 \\ a_2 & a_1 \end{pmatrix} , $$
where $a_1,a_2$ are integers such that $a_1\in\{-1,-2\}$ and $1\leq|a_2|\leq 1-a_1$.
\item If each slim vertex of $\mathfrak{f}$ has exactly one fat neighbor, then $S(\mathfrak{f})$ is one of the following.
$$\begin{pmatrix} -1 & 2 \\ 2 & -1 \end{pmatrix} ,\begin{pmatrix} -1 & -2 \\ -2 & -1 \end{pmatrix},\begin{pmatrix} -1 & -1&1 \\ -1 & -1&1\\1&1&-1 \end{pmatrix},\begin{pmatrix} -1 & -1&-1 \\ -1 & -1&-1\\-1&-1&-1 \end{pmatrix} \text{ and }$$
$$\begin{pmatrix} -1 & 1&1 \\ 1 & -1&0\\1&0&-1 \end{pmatrix},\begin{pmatrix} -1 & -1&1 \\ -1 & -1&0\\1&0&-1 \end{pmatrix} ,\begin{pmatrix} -1 & -1&-1 \\ -1 & -1&0\\-1&0&-1\end{pmatrix}. $$
\end{enumerate}
In particular, $\lambda_{\min}(\mathfrak{f})\leq-1-\sqrt{2}$ and the set $\mathcal{F}(-2)$ is finite.
\end{theorem}
\begin{proof}
$\rm{(i)}$ Assume that $\mathfrak{f}$ contains a slim vertex with at least $2$ fat neighbors. Since $\mathfrak{f}\in\mathcal{F}(-2)$, we have the special matrix $S(\mathfrak{f})$ is
$$(-3)\text{ or }\begin{pmatrix} -2 & a_2 \\ a_2 & a_1 \end{pmatrix} , $$
where $a_1,a_2$ are integers such that $a_1\in\{-1,-2\}$ and $1\leq|a_2|\leq 1-a_1$. By direct computation, we obtain that $\lambda_{\min}(\mathfrak{f})<-1-\sqrt{2}$.\par
$\rm{(ii)}$ Note that all of the diagonal entries of $S(\mathfrak{f})$ are $-1$, as each slim vertex of $\mathfrak{f}$ has exactly one fat neighbor. If $S(\mathfrak{f})$ contains an entry with absolute value at least $2$, then $$S(\mathfrak{f})=\begin{pmatrix} -1 & b \\ b & -1 \end{pmatrix},$$
where $b\in\{-2,2\}$ and $\lambda_{\min}(\mathfrak{f})=-3$. 
If all the entries of $S(\mathfrak{f})$ have absolute value at most $1$, then there exists a signed graph $(H,\tau)$ with adjacency matrix $A(H,\tau)=S(\mathfrak{f})+\mathbf{I}$. Since the smallest eigenvalue of $(H,\tau)$ satisfies that $\lambda_{\min}(H,\tau)=\lambda_{\min}(\mathfrak{f})+1<-1$, then we obtain that $\lambda_{\min}(H,\tau)\leq-\sqrt{2}$, by Lemma \ref{GMST -sqrt{2}}. Note that for any signed graph with smallest eigenvalue at most $-\sqrt{2}$, it contains an induced subgraph switching equivalent to either $(K_3,-)$ or $(K_{1,2},+)$. Thus, $(H,\tau)$ is switching equivalent to either $(K_3,-)$ or $(K_{1,2},+)$ and the special matrix $S(\mathfrak{f})=A(H,\tau)-\mathbf{I}$. By direct computation, $\lambda_{\min}(\mathfrak{f})\leq \lambda_{\min}(H,\tau)-1\leq-1-\sqrt{2}$.\par

Since for any $\mathfrak{f}\in \mathcal{F}(-2)$, $\mathfrak{f}$ has at most $3$ slim vertices and at most $4$ fat vertices, which implies that $\mathcal{F}(-2)$ is finite. This completes the proof.
\end{proof}

\begin{corollary}\label{-2 to -1-sqrt2 for fhsg}
    Let $\mathfrak{h}$ be a fat Hoffman signed graph with smallest eigenvalue $\lambda_{\min}(\mathfrak{h})$. If $\lambda_{\min}(\mathfrak{h})<-2$, then $\lambda_{\min}(\mathfrak{h})\leq -1-\sqrt{2}$.
\end{corollary}
\begin{proof}
    This follows immediately from Theorem \ref{min forbidden finite for -2}.
\end{proof}

Now we prove Theorem \ref{thm for min ev >-1-sqrt{2}}.

\noindent
\textbf{Proof of Theorem \ref{thm for min ev >-1-sqrt{2}}}. For $\lambda\in(-2,-1]$, let $d_{\lambda}'= f(\lambda)$, where $f(\lambda)$ is as defined in Theorem \ref{GMST result on -2}, then $(G,\sigma)$ is switching equivalent to a positive clique and $\lambda_{\min}(G,\sigma)=-1$. Assume that $(G,\sigma)$ is obtained from $(K_{|V(G)|},+)$ by switching on $U$ for some $U\in V(G)$. Define the diagonal matrix $D$ as
$$D_{xx}=\left\{\begin{array}{ll}
    -1, & \text{if } x\in U \\
    1,  & \text{otherwise},
  \end{array}\right.$$
then $D=D^{-1}=D^T$.
Note that $A(G,\sigma)=D(A(K_{|V(G)|},+))D=D(\mathbf{J}-\mathbf{I})D$, where $\mathbf{J}$ is the all-ones matrix.
Let $N$ be the matrix with order $|V(G)|$ and with $1$ on the first row and $0$ on the rest, then $\mathbf{J}=N^TN$. Thus, $A(G,\sigma)+\mathbf{I}=D\mathbf{J}D=DN^TND=(ND)^TND$, and $(G,\sigma)$ is $1$-integrable.\par

Now we may assume that $-1-\sqrt{2}<\lambda\leq-2$ and $\lambda_{\min}(G,\sigma)\leq -2$. To obtain $d_{\lambda}'$, we need to look at the set $\mathcal{F}(-2)$. 
For any $\mathfrak{f}\in \mathcal{F}(-2)$, $\lambda_{\min}(\mathfrak{f})\leq -1-\sqrt{2}$, by Corollary \ref{-2 to -1-sqrt2 for fhsg}. Thus, there exists a positive integer $p_{\mathfrak{f}}$ such that $\lambda_{\min}(G(\mathfrak{f},p_{\mathfrak{f}}))<\lambda$, by Theorem \ref{Hoffman-Ostrowski Theorem of signed case}. Since $\mathcal{F}(-2)$ is finite, we may denote $v_f=\max\{|V_f(\mathfrak{f})\mid \mathfrak{f}\in\mathcal{F}(-2)\}$, $v_s=\max\{|V_s(\mathfrak{f})\mid \mathfrak{f}\in\mathcal{F}(-2)\}$ and $p=\max\{p_{\mathfrak{f}}\mid \mathfrak{f}\in\mathcal{F}(-2)\}$. 
Let $m_{\lambda}=\min\{m\mid\lambda_{\min}(\widetilde{K}_{2m}^{(0)})<\lambda,\lambda_{\min}(\widetilde{K}_{2m}^{(-)})<\lambda\}$ and $q=q(m_{\lambda},v_f,v_s,p)$ be the integer as defined in Proposition \ref{associatedProp}.\par
Let $n=\max\{q,q_{\lambda}\}$ and $d_{\lambda}'=\max\{d(\lambda,n),120\}$, where $q_{\lambda}$ and $d(\lambda,n)$ are as defined in Theorem \ref{structure th for fixed min ev}. Let $\mathfrak{g}=\mathfrak{g}(G,\sigma,m_{\lambda},n)$ be an associated Hoffman signed graph of $(G,\sigma)$ with respect to the pair $(m_{\lambda},n)$. 
Since $\lambda_{\min}(G,\sigma)\geq \lambda$, $\lambda_{\min}(\widetilde{K}_{2m_{\lambda}}^{(0)})<\lambda$, $\lambda_{\min}(\widetilde{K}_{2m_{\lambda}}^{(-)})<\lambda$ and $\lambda_{\min}(G(\mathfrak{f},p))\leq\lambda_{\min}(G(\mathfrak{f},p_{\mathfrak{f}}))<\lambda$, then $(G,\sigma)$ is $\{\widetilde{K}_{2m_{\lambda}}^{(0)},\widetilde{K}_{2m_{\lambda}}^{(-)}\}$-switching-free and $(G,\sigma)$ does not contain $G(\mathfrak{f},p)$ as an induced subgraph for any $\mathfrak{f}\in \mathcal{F}(-2)$. 
Thus, by Proposition \ref{associatedProp}, $\mathfrak{g}$ does not contain any Hoffman signed graph in $\mathcal{F}(-2)$ as an induced subgraph. As $d_{\lambda}'\geq d(\lambda,n)$, Theorem \ref{structure th for fixed min ev} $(\rm{i})$ implies that each vertex of $(G,\sigma)$ has at least one fat neighbor, consequently, $\mathfrak{g}$ is fat. Now we show that $\lambda_{\min}(\mathfrak{g})\geq -2$. Suppose, for a contradiction, $\lambda_{\min}(\mathfrak{g})<-2$, then $\mathfrak{g}$ contains a minimal fat Hoffman signed induced subgraph with smallest eigenvalue less than $-2$. This contradicts that $\mathfrak{g}$ does not contain any member of $\mathcal{F}(-2)$. It follows that $\lambda_{\min}(G,\sigma)=-2$, as by Lemma \ref{sub-minev}. Since $(G,\sigma)$ contains at least $1+d_{\lambda}\geq 121$ vertices, signed graph $(G,\sigma)$ is $1$-integrable, by Theorem \ref{BCKW result on -2}.\par
 This completes the proof. \qed

\section*{Acknowledgements}
J.H. Koolen is partially supported by the National Natural Science Foundation of China (No. 12471335), and the 
Anhui Initiative in Quantum Information Technologies (No. AHY150000).
Q. Yang is supported by the National Natural Science Foundation of China (No. 12401460).

\section*{Author Contributions}
All authors contributed equally in this work.

\section*{Financial disclosure}

None reported.

\section*{Conflict of interest}

The authors declare no potential conflict of interests.

\bibliography{CKLYref}

@article {Balla2025equiangular,
    AUTHOR = {Balla, I.},
     TITLE = {Equiangular lines via matrix projection},
   JOURNAL = {Adv. Math.},
  FJOURNAL = {Advances in Mathematics},
    VOLUME = {482},
      YEAR = {2025},
     PAGES = {110620},
      ISSN = {0001-8708,1090-2082},
   MRCLASS = {14C30 (05C50)},
  MRNUMBER = {4976612},
       DOI = {10.1016/j.aim.2025.110620},
       URL = {https://doi.org/10.1016/j.aim.2025.110620},
}

@article {Bannai1983an,
    AUTHOR = {Bannai, E. and Bannai, E. and Stanton, D.},
     TITLE = {An upper bound for the cardinality of an {$s$}-distance subset in real {E}uclidean space. {II}},
   JOURNAL = {Combinatorica},
  FJOURNAL = {Combinatorica. An International Journal of the J\'anos Bolyai
              Mathematical Society},
    VOLUME = {3},
      YEAR = {1983},
    NUMBER = {2},
     PAGES = {147--152},
      ISSN = {0209-9683},
   MRCLASS = {52A37 (05B25 51D20)},
  MRNUMBER = {726452},
MRREVIEWER = {K.\ J.\ Falconer},
       DOI = {10.1007/BF02579288},
       URL = {https://doi.org/10.1007/BF02579288},
}

@article {belardo2018open,
    AUTHOR = {Belardo, F. and Cioab\u{a}, S. M. and Koolen,
              J. H. and Wang, J.-F.},
     TITLE = {Open problems in the spectral theory of signed graphs},
   JOURNAL = {Art Discrete Appl. Math.},
  FJOURNAL = {The Art of Discrete and Applied Mathematics},
    VOLUME = {1},
      YEAR = {2018},
    NUMBER = {2},
     PAGES = {\#P2.10},
      ISSN = {2590-9770},
   MRCLASS = {05C50 (05C22)},
  MRNUMBER = {3997096},
MRREVIEWER = {Leila\ Parsaei Majd},
       DOI = {10.26493/2590-9770.1286.d7b},
       URL = {https://doi.org/10.26493/2590-9770.1286.d7b},
}

@phdthesis{Blokhuis1984Few,
    author = {Blokhuis, A.},
    title = {Few-distance sets},
    school = {Technische Hogeschool Eindhoven},
    year = {1984},
  url={https://doi.org/10.6100/IR53747}
}

@article{Gavrilyuk2021signed,
author = {Gavrilyuk, A. L. and Munemasa, A. and Sano, Y. and Taniguchi, T.},
title = {Signed analogue of line graphs and their smallest eigenvalues},
journal = {J. Graph Theory},
fjournal = {Journal of Graph Theory},
volume = {98},
number = {2},
pages = {309-325},
year = {2021},
doi = {https://doi.org/10.1002/jgt.22699},
url = {https://onlinelibrary.wiley.com/doi/abs/10.1002/jgt.22699},
eprint = {https://onlinelibrary.wiley.com/doi/pdf/10.1002/jgt.22699},
}

@article{glazyrin2018upper,
    AUTHOR = {Glazyrin, A. and Yu, W.-H.},
     TITLE = {Upper bounds for {$s$}-distance sets and equiangular lines},
   JOURNAL = {Adv. Math.},
  FJOURNAL = {Advances in Mathematics},
    VOLUME = {330},
      YEAR = {2018},
     PAGES = {810--833},
      ISSN = {0001-8708,1090-2082},
   MRCLASS = {52C10},
  MRNUMBER = {3787558},
MRREVIEWER = {Konrad\ J.\ Swanepoel},
       DOI = {10.1016/j.aim.2018.03.024},
       URL = {https://doi.org/10.1016/j.aim.2018.03.024},
}

@book{godsil2013,
    AUTHOR = {Godsil, C. and Royle, G.},
     TITLE = {Algebraic graph theory},
    SERIES = {Graduate Texts in Mathematics},
    VOLUME = {207},
 PUBLISHER = {Springer-Verlag, New York},
      YEAR = {2001},
     PAGES = {xx+439},
      ISBN = {0-387-95241-1; 0-387-95220-9},
   MRCLASS = {05-02 (05C50 05E30)},
  MRNUMBER = {1829620},
MRREVIEWER = {Robin\ J.\ Wilson},
       DOI = {10.1007/978-1-4613-0163-9},
       URL = {https://doi.org/10.1007/978-1-4613-0163-9},
}

@incollection {hoffman1973on,
    AUTHOR = {Hoffman, A. J.},
     TITLE = {On spectrally bounded graphs},
 BOOKTITLE = {A survey of combinatorial theory},
     PAGES = {277--283},
 PUBLISHER = {North-Holland, Amsterdam-London},
      YEAR = {1973},
   MRCLASS = {05C99},
  MRNUMBER = {376441},
MRREVIEWER = {A.\ J.\ Schwenk},
}

@article{hoffman1977on,
    AUTHOR = {Hoffman, A. J.},
     TITLE = {On graphs whose least eigenvalue exceeds {$-1-\sqrt{2}$}},
   JOURNAL = {Linear Algebra Appl.},
  FJOURNAL = {Linear Algebra and its Applications},
    VOLUME = {16},
      YEAR = {1977},
    NUMBER = {2},
     PAGES = {153--165},
      ISSN = {0024-3795,1873-1856},
   MRCLASS = {05C99},
  MRNUMBER = {469826},
MRREVIEWER = {A.\ J.\ Schwenk},
       DOI = {10.1016/0024-3795(77)90027-1},
       URL = {https://doi.org/10.1016/0024-3795(77)90027-1},
}

@article {hoffman1977onsigned,
    AUTHOR = {Hoffman, A. J.},
     TITLE = {On signed graphs and gramians},
   JOURNAL = {Geometriae Dedicata},
  FJOURNAL = {Geometriae Dedicata},
    VOLUME = {6},
      YEAR = {1977},
    NUMBER = {4},
     PAGES = {455--470},
   MRCLASS = {15A45 (05C99)},
  MRNUMBER = {463211},
MRREVIEWER = {R.\ A.\ Brualdi},
       DOI = {10.1007/BF00147783},
       URL = {https://doi.org/10.1007/BF00147783},
}

@article {Jang2014on,
    AUTHOR = {Jang, H. J. and Koolen, J. and Munemasa, A. and Taniguchi, T.},
     TITLE = {On fat {H}offman graphs with smallest eigenvalue at least
              {$-3$}},
   JOURNAL = {Ars Math. Contemp.},
  FJOURNAL = {Ars Mathematica Contemporanea},
    VOLUME = {7},
      YEAR = {2014},
    NUMBER = {1},
     PAGES = {105--121},
      ISSN = {1855-3966,1855-3974},
   MRCLASS = {05C50 (05C76)},
  MRNUMBER = {3047614},
       DOI = {10.26493/1855-3974.262.a9d},
       URL = {https://doi.org/10.26493/1855-3974.262.a9d},
}

@article{kim2016a,
    AUTHOR = {Kim, H. K. and Koolen, J. H. and Yang, J. Y.},
     TITLE = {A structure theory for graphs with fixed smallest eigenvalue},
   JOURNAL = {Linear Algebra Appl.},
  FJOURNAL = {Linear Algebra and its Applications},
    VOLUME = {504},
      YEAR = {2016},
     PAGES = {1--13},
      ISSN = {0024-3795,1873-1856},
   MRCLASS = {05C50 (05C75)},
  MRNUMBER = {3502527},
MRREVIEWER = {David\ Burns},
       DOI = {10.1016/j.laa.2016.03.044},
       URL = {https://doi.org/10.1016/j.laa.2016.03.044},
}

@article {Koolen2021recent,
    AUTHOR = {Koolen, J. H. and Cao, M.-Y. and Yang, Q.},
     TITLE = {Recent progress on graphs with fixed smallest adjacency
              eigenvalue: a survey},
   JOURNAL = {Graphs Combin.},
  FJOURNAL = {Graphs and Combinatorics},
    VOLUME = {37},
      YEAR = {2021},
    NUMBER = {4},
     PAGES = {1139--1178},
      ISSN = {0911-0119,1435-5914},
   MRCLASS = {05C50 (05C22 05C75 05D99 05E30 11H06)},
  MRNUMBER = {4280318},
       DOI = {10.1007/s00373-021-02296-8},
       URL = {https://doi.org/10.1007/s00373-021-02296-8},
}

@article{vanlint1966,
    AUTHOR = {van Lint, J. H. and Seidel, J. J.},
     TITLE = {Equilateral point sets in elliptic geometry},    
   JOURNAL = {Indag. Math.},
  FJOURNAL = {Indagationes Mathematicae},
    VOLUME = {28},
      YEAR = {1966},
     PAGES = {335--348},
   MRCLASS = {52.50},
  MRNUMBER = {200799},
MRREVIEWER = {L.\ M.\ Blumenthal},
}

@article{woo1995on,
    AUTHOR = {Woo, R. and Neumaier, A.},
     TITLE = {On graphs whose smallest eigenvalue is at least {$-1-\sqrt2$}},
   JOURNAL = {Linear Algebra Appl.},
  FJOURNAL = {Linear Algebra and its Applications},
    VOLUME = {226/228},
      YEAR = {1995},
     PAGES = {577--591},
      ISSN = {0024-3795,1873-1856},
   MRCLASS = {05C50},
  MRNUMBER = {1344587},
MRREVIEWER = {Qiao\ Li},
       DOI = {10.1016/0024-3795(95)00245-M},
       URL = {https://doi.org/10.1016/0024-3795(95)00245-M},
}

@article {ramsey1930,
author = {Ramsey, F. P.},
title = {On a Problem of Formal Logic},
journal = {Proceedings of the London Mathematical Society},
volume = {s2-30},
number = {1},
pages = {264-286},
doi = {https://doi.org/10.1112/plms/s2-30.1.264},
url = {https://londmathsoc.onlinelibrary.wiley.com/doi/abs/10.1112/plms/s2-30.1.264},
eprint = {https://londmathsoc.onlinelibrary.wiley.com/doi/pdf/10.1112/plms/s2-30.1.264},
year = {1930}
}

\end{document}